\documentclass[10pt, a4paper]{amsart}
\usepackage[
bookmarks=true,
colorlinks=true,
linkcolor=blue,
urlcolor=blue,
bookmarks=true,
citecolor=blue,
hyperindex=true
]{hyperref}
\usepackage{amsmath}
\numberwithin{equation}{section}
\usepackage{amsthm}
\usepackage{amssymb}
\usepackage[abbrev]{amsrefs}
\usepackage{bm}

\usepackage{enumerate}
\usepackage{enumitem}
\usepackage{subfigure}
\usepackage{graphicx}

\makeatletter
\let\orgdescriptionlabel\descriptionlabel
\renewcommand*{\descriptionlabel}[1]{%
  \let\orglabel\label
  \let\label\@gobble
  \phantomsection
  \edef\@currentlabel{#1}%
  \let\label\orglabel
  \orgdescriptionlabel{#1}%
}
\makeatother

\usepackage{indentfirst}
\usepackage{geometry}
\usepackage{marginnote}

\newtheorem{theorem}{Theorem}[section]
\newtheorem{theoremx}{Theorem}

\newtheorem{proposition}[theorem]{Proposition}
\newtheorem{example}[theorem]{Example}

\newtheorem{corollary}[theorem]{Corollary}
\newtheorem{lemma}[theorem]{Lemma}

\newtheorem{remark}[theorem]{Remark}
\newtheorem*{remark*}{Remark}
\newtheorem{definition}[theorem]{Definition}

\newcommand*{\cE}{\mathcal{E}}
\newcommand*{\fF}{\mathfrak{F}}

\newcommand*{\cK}{\mathcal{K}}
\newcommand*{\M}{\mathcal{M}}

\newcommand*{\cX}{\mathcal{X}}

\newcommand*{\Rb}{\mathbb{R}}
\newcommand*{\Zb}{\mathbb{Z}}
\newcommand*{\Nb}{\mathbb{N}}

\newcommand*{\dif}{\mathop{}\!\mathrm{d}}

\newcommand*{\Merg}{\mathcal{M}_{\mathrm{erg}}}
\newcommand*{\htop}{h_{\mathrm{top}}}
\newcommand*{\lh}{\underline{h}}
\newcommand*{\ve}{\varepsilon}

\newcommand*{\Orb}{\mathrm{Orb}}

\newcommand*{\Trans}{\mathrm{Trans}}

\newcommand*{\diam}{\mathrm{diam}}

\newcommand*{\supp}{\mathrm{supp}}
\newcommand*{\Sing}{\mathrm{Sing}}
\newcommand*{\Int}{\mathrm{Int}}

\newcommand{\Wuloc}{W^u_{\mathrm{loc}}}
\newcommand{\Wsloc}{W^s_{\mathrm{loc}}}

\let\oldtocsection=\tocsection

\let\oldtocsubsection=\tocsubsection

\let\oldtocsubsubsection=\tocsubsubsection

\renewcommand{\tocsection}[2]{\hspace{0em}\oldtocsection{#1}{#2}}
\renewcommand{\tocsubsection}[2]{\hspace{1em}\oldtocsubsection{#1}{#2}}
\renewcommand{\tocsubsubsection}[2]{\hspace{2em}\oldtocsubsubsection{#1}{#2}}

\title[]{Topological entropy of saturated sets for Lorenz attractors}
\author{Zheng Huang}
\address{School of Mathematical Sciences, CMA-Shanghai, Shanghai Jiao Tong University, \newline \hspace*{\parindent} Shanghai, 200240, China}
\email{hzxs280457@sjtu.edu.cn}

\author{Xueting Tian}
\address{School of Mathematical Sciences, Fudan University, Shanghai, 200433, China}
\email{xuetingtian@fudan.edu.cn}

\author{Xiaodong Wang}
\address{School of Mathematical Sciences, CMA-Shanghai, Shanghai Jiao Tong University, \newline \hspace*{\parindent} Shanghai, 200240, China}
\email{xdwang1987@sjtu.edu.cn}

\author{Xutong Zhao}
\address{School of Mathematical Sciences, CMA-Shanghai, Shanghai Jiao Tong University, \newline \hspace*{\parindent} Shanghai, 200240, China}
\email{xutongzhao@sjtu.edu.cn}

\thanks{Z. Huang, X. Wang, and X. Zhao were supported by the National Key R\&D Program of China (No. 2021YFA1001900).  }
\thanks{X. Tian was supported by the National Natural Science Foundation of China (No. 12471182).}

\begin{document}
\begin{abstract}
    We study the topological entropy of saturated sets for continuous flows,
    with emphasis on singular hyperbolic dynamics. 
    For a $C^r$-generic
    ($r \geq 2$) vector field $X$ carrying a geometric Lorenz attractor $\Lambda$,
    we prove that for every nonempty compact connected set $\mathcal{K} \subseteq \M(\Phi,\Lambda)$ and every open set $U \subseteq M$ such that $U \cap \Lambda \neq \emptyset$, 
    \[
        \htop\left(\Phi, G_{\mathcal{K}}(\Phi) \cap \Trans(\Phi,\Lambda) \cap U\right)
        = \inf_{\mu \in \mathcal{K}} h_\mu(\phi_{1}),
    \]
    where $\Phi=(\phi_{t})_{t \in \mathbb{R}}$ is the flow induced by $X$. Here, \(\Trans(\Phi,\Lambda)\) denotes the set of points whose positive orbits are dense in \(\Lambda\). The same formula holds $C^1$-generically for singular hyperbolic
    attractors.  In contrast with the uniformly hyperbolic case, we also prove that, for vector fields in a dense subset of $\mathcal{X}^{r}(M^3)$, geometric Lorenz attractors admit a positive-entropy invariant measure with no generic points. An analogous conclusion holds for vector fields in a dense subset of $\mathcal{X}^{1}(M)$ whose non-trivial singular attractors contain singularities and satisfy $\dim E^{cu}=2$. These examples show that the discrete-time saturated entropy theory does not extend directly to singular flows.
\end{abstract}

\maketitle
\tableofcontents

\section{Introduction}
Multifractal analysis of dynamical systems studies the size and complexity of sets of points with asymptotic statistical behaviors. A central object in this theory is the empirical measure, which records the average distribution of an orbit over time.
Let $\Phi=(\phi_t)_{t\in\mathbb R}$ be a continuous flow on a compact metric space $K$, and let $\Lambda\subseteq K$ be a compact $\Phi$-invariant set. Let $\mathcal M(\Lambda)$ denote the space of Borel probability measures supported on $\Lambda$, endowed with the weak* topology.
We denote by $\M(\Phi,\Lambda)$ and $\Merg(\Phi,\Lambda)$ the sets of $\Phi$-invariant and $\Phi$-ergodic probability measures supported on $\Lambda$, respectively. For $x\in\Lambda$ and $t>0$, the \emph{empirical measure} of $x$ up to time $t$ is defined by
$$ \mathcal{E}_t(x,\Phi)=\frac1t\int_0^t\delta_{\phi_s(x)}\,ds, $$
where $\delta_y$ denotes the Dirac probability measure supported at $y$.
Let $V_\Phi(x)$ be the set of accumulation points of $\mathcal{E}_t(x,\Phi)$. For a nonempty compact connected subset $\mathcal{K} \subset \mathcal M(\Phi,\Lambda)$, the set
\[
    G_{\mathcal{K}}(\Phi)=\{x\in\Lambda: V_\Phi(x)=\mathcal{K}\}
\]
is called the \emph{saturated set} of $\mathcal{K}$.

When $\mathcal{K}=\{\mu\}$, the set $G_{\mathcal{K}}=G_{\mu}$ consists of generic points of the invariant measure $\mu$; for general $\mathcal{K}$, it consists of points whose empirical measures oscillate exactly inside $\mathcal{K}$.

The entropy theory of noncompact sets goes back to Bowen \cite{bowen1973topological}, who introduced a Carathéodory-type definition of topological entropy for arbitrary subsets. This point of view has become a central tool in the study of irregular sets, level sets, and saturated sets. 
For discrete-time systems with specification (more generally, the $g$-almost product property) together with the uniform separation property, Pfister and Sullivan \cite{pfister2007topological} established a general entropy formula for saturated sets: for every
nonempty compact connected set $\mathcal K$ of invariant measures,
\[
h_{\mathrm{top}}(T,G_{\mathcal K})
=
\inf_{\mu\in\mathcal K} h_\mu(T).
\]
Related variational principles for noncompact sets and measure-theoretical lower entropy were developed by Feng and Huang \cite{feng2012variational}. 
In the nonuniformly hyperbolic setting, Liang, Liao, Sun and Tian \cite{liang2016variational} established an analogous variational equality for compact connected families of invariant measures with a common hyperbolic rate. 

Huang, Tian and Wang \cite{huang2019transitively} introduced the transitively-saturated property, which strengthens the saturated property by requiring the points in the saturated set to be transitive. Under the $g$-almost product and uniform separation properties, they obtained the corresponding entropy formula for transitively-saturated sets.

For flows, however, several additional difficulties arise. First, the empirical measures are defined by continuous-time averages, while many entropy estimates, such as measure-theoretic entropy, are more naturally formulated for the time-one map.
Second, the relation of invariant measures between the flow and its time-one map is subtle, and generic points for the two systems need not behave in the same way. For example, positive entropy does not necessarily guarantee the existence of generic points for the flow. 

When specializing to smooth flows generated by vector fields, we write $\Phi^X=(\phi_t^X)_{t\in\mathbb R}$ for the flow generated by a vector field $X$. We prove in Theorem \ref{thm:C} that there exists a dense subset $\mathcal D^r\subset \mathfrak X^r(M^3)$, $r\ge 2$, such that if $\Lambda$ is a geometric Lorenz attractor of $X\in\mathcal D^r$, then there exists $\mu\in\mathcal M(\Phi^X,\Lambda)$ with $h_\mu(\phi^X_{1})>0$ and $G_\mu(\Phi^X)=\emptyset$. Thus, a direct application of the discrete-time saturated entropy theory is not sufficient.

The purpose of this paper is to establish a saturated entropy formula for singular hyperbolic flows, with special emphasis on geometric Lorenz attractors and singular hyperbolic attractors. 

The Lorenz system was introduced by Lorenz as a simplified model for atmospheric convection and has become a classical example of chaotic dynamics \cite{lorenz1963deterministic}. Later, Guckenheimer \cite{guckenheimer1976strange} and Afraimovich, Bykov and Shil'nikov \cite{afraimovich1977origin} introduced geometric models for the Lorenz attractor. These models, now known as geometric Lorenz attractors, were further studied by Guckenheimer and Williams and by Williams, who investigated their structural stability and geometric structure \cites{guckenheimer1979structural,williams1979structure}. Singular hyperbolicity was introduced by Morales, Pacifico and Pujals \cites{morales1999singular, morales2004robust} as a robust framework for Lorenz-like dynamics, and was later developed in higher-dimensional settings in \cites{li2005robustly, metzger2008sectional, zhu2008indices}. More recently, the thermodynamic formalism and equilibrium states for Lorenz and singular hyperbolic attractors have also been studied in \cite{pacifico2025equilibrium}. These systems form a natural class beyond uniformly hyperbolic flows, since singularities may be present, and the dynamics is only partially hyperbolic in a singular sense.

Our first main result concerns geometric Lorenz attractors. We prove in Theorem \ref{thm:A} that for every $r\ge 2$, there exists a residual subset $\mathcal R^r\subset \mathfrak X^r(M^3)$, such that if $\Lambda$ is a geometric Lorenz attractor of $X\in\mathcal R^r$, then for every nonempty compact connected set $\mathcal{K} \subset \mathcal M(\Phi^X,\Lambda)$ and every open set $U \subseteq M$ such that $U \cap \Lambda \neq \emptyset$,
\[
    h_{\mathrm{top}}\bigl(\Phi^X,G_{\mathcal{K}}(\Phi^X)\cap
    \mathrm{Trans}(\Phi^X,\Lambda)\cap U\bigr)
    =\inf_{\mu\in \mathcal{K}}h_\mu(\phi^X_1).
\]
Here and below, \(\Trans(\Phi,\Lambda)\) denotes transitive points of $\Phi$ in $\Lambda$. That is,
\[\Trans\left(\Phi,\Lambda\right) = \left\{x \in \Lambda:\omega_{\Phi}(x) = \Lambda\right\}.\]
Thus, the saturated set has the expected entropy even after restricting to transitive points and to an arbitrary open subset of the attractor.

Our second main result establishes the same formula for singular hyperbolic attractors in the $C^1$ generic setting, namely Theorem \ref{thm:B}.
These results extend the saturated entropy principle from uniformly hyperbolic and specification-type systems to a class of singular hyperbolic flows.

As an application of Theorems~\ref{thm:A} and~\ref{thm:B}, we give an alternative proof of the multifractal results for Lorenz-type attractors established in \cite{shi2023multifractal}. In that work, the irregular set (where time averages do not converge) and level sets (where averages equal a given value) were studied; see Theorem~\ref{thm:2.1}. With a suitable choice of $\mathcal{K} \subset \mathcal{M}(\Phi,\Lambda)$, we calculate the entropy of irregular sets and level sets using our entropy formula.

\subsection*{Main results}
Let $\cX^r(M^d)$, $r \geq 1$, denote the space of $C^r$-vector fields on a $d$-dimensional compact Riemannian manifold $M$ endowed with the $C^r$ topology. Let $\Nb, \Nb^+, \Zb, \Rb$ denote the set of natural numbers, positive integers, integers, and real numbers, respectively. For $X\in\cX^r(M)$, let $\Phi^X=(\phi_t^X)_{t\in\mathbb R}$ denote the $C^r$ flow generated by $X$. When $X$ is clear, we simply write $\Phi=(\phi_t)_{t\in\mathbb R}$. For $\mu\in\M(\Phi,\Lambda)$, the measure-theoretic entropy of the flow is understood throughout as $h_\mu(\phi_1)$, the entropy of its time-one map.

We use the notation $\mathcal{E}_t(x,\Phi)$, $V_\Phi(x)$ and $G_{\mathcal{K}}(\Phi)$ introduced above. Throughout the main results, $\mathcal{K}$ will be a nonempty compact connected subset of $\mathcal M(\Phi,\Lambda)$.

Our first and main result asserts that, for a $C^r$-generic vector field, the saturated entropy formula holds on a geometric Lorenz attractor, and persists after intersecting with the transitive points and with an arbitrary open set.

\begin{theoremx}\label{thm:A}
    For every $r \geq 2$, there exists a residual subset $\mathcal{R}^r \subset \cX^r\left(M^3\right)$ such that, if $\Lambda$ is a geometric Lorenz attractor of $X \in \mathcal{R}^r$, then for any nonempty compact connected subset $\mathcal{K}$ of $\M(\Phi, \Lambda)$, we have
    \[\htop\left(\Phi, G_{\mathcal{K}}\left(\Phi\right)  \cap \Trans(\Phi,\Lambda) \cap U\right) = \inf\left\{h_\mu\left(\phi_1\right) : \mu \in \mathcal{K} \right\}\]
    for any nonempty open set $U \subseteq M$ with $U \cap \Lambda \ne \emptyset$.
\end{theoremx}

The same conclusion persists, $C^1$-generically, in the broader class of singular hyperbolic attractors.

\begin{theoremx}\label{thm:B}
    There exists a residual subset $\mathcal{R} \subset \cX^1\left(M\right)$ such that, if $\Lambda$ is a non-trivial singular hyperbolic attractor of $X \in \mathcal{R}$, then for any nonempty compact connected subset $\mathcal{K}$ of $\M(\Phi, \Lambda)$, we have
    \[\htop\left(\Phi, G_{\mathcal{K}}\left(\Phi\right)  \cap \Trans(\Phi,\Lambda) \cap U \right) = \inf\left\{h_\mu\left(\phi_1\right) : \mu \in \mathcal{K} \right\}\]
    for any nonempty open set $U \subseteq M$ with $U \cap \Lambda \ne \emptyset$.
\end{theoremx}

\begin{remark}
    In the systems of the above two theorems, the set $G_{\mathcal{K}}\left(\Phi\right)  \cap \Trans(\Phi,\Lambda) \cap U$ remains non-empty even if the right-hand side of the equality vanishes.
\end{remark}

In \cite{shi2025connecting}*{Theorem B}, the authors give characterizations of the space of invariant measures for $C^r (r\geq2)$-generic and $C^r$-dense Lorenz attractors. They proved that for $C^r (r\geq2)$-generic  Lorenz attractors, periodic measures are dense in the space of invariant measures and the set of ergodic measures is path-connected while for $C^r$-dense Lorenz attractors, the singular measure is an isolated ergodic measure and thus the set of ergodic measures is not connected, which gives a sharp contrast (a similar property holds for $C^1$ singular hyperbolic attractors)\cite{shi2025connecting}*{Theorem B.1}. Based on the generic part, the authors of \cite{shi2023multifractal} prove the horseshoe approximation property for  $C^r (r\geq2)$-generic  Lorenz attractors and $C^1$-generic singular hyperbolic attractors and then study the multifractal analysis of such systems. The horseshoe approximation property plays a key role in the proofs of Theorems \ref{thm:A} and \ref{thm:B}.

In sharp contrast with the uniformly hyperbolic setting, our final two results show that generic points may fail to exist even for measures of positive entropy. This is precisely the obstruction announced in the introduction, and it explains why the discrete-time saturated theory cannot be applied directly.

The main idea originates in \cite{shi2025connecting}, where the singular measure is isolated in the set of ergodic measures for $C^r (r\geq 2)$ dense Lorenz attractors and $C^1$ dense singular hyperbolic attractors with a codimension two splitting.
\begin{theoremx}\label{thm:C}
    For every $r \geq 2$, there exists a dense subset $\mathcal{D}^r \subset \cX^r\left(M^3 \right)$ such that, if $\Lambda$ is a geometric Lorenz attractor of $X \in \mathcal{D}^r$, then there exists $\mu \in \M(\Phi, \Lambda)$ such that $h_\mu\left(\phi_1\right) > 0$ and $G_\mu(\Phi) = \emptyset$.
\end{theoremx}

\begin{theoremx}\label{thm:D}
    There exists a dense subset $\mathcal{D} \subset \cX^1\left(M\right)$ such that, if $\Lambda$ is a non-trivial singular hyperbolic attractor of $X \in \mathcal{D}$ satisfying $\Sing(\Lambda)\neq\emptyset$ and  whose singular hyperbolic splitting $T_{\Lambda}M=E^{ss} \oplus E^{cu}$ satisfies $\dim(E^{cu})=2$, then there exists $\mu \in \M(\Phi, \Lambda)$ such that $h_\mu\left(\phi_1\right) > 0$ and $G_\mu(\Phi) = \emptyset$.
\end{theoremx}

We conclude this section with a brief comparison with the multifractal results of
\cite{shi2023multifractal}. For \(g\in C(\Lambda,\mathbb{R})\), define the
irregular set and the level sets by
\[
    I_g=
    \left\{
    x\in\Lambda:
    \lim_{T\to\infty}
    \frac1T\int_0^T g(\phi_t(x))\,dt
    \ \text{does not exist}
    \right\}
\]
and
\[
    R_g(a)=
    \left\{
    x\in\Lambda:
    \lim_{T\to\infty}
    \frac1T\int_0^T g(\phi_t(x))\,dt
    =
    a
    \right\},
    \qquad a\in\mathbb{R},
\]
respectively.

Under the generic hypotheses considered in \cite{shi2023multifractal},
the multifractal results there give a dichotomy for the irregular set:
either \(I_g=\varnothing\), or \(I_g\) is residual in \(\Lambda\) and carries full topological entropy. Moreover, in the latter case, the nontrivial level sets \(R_g(a)\) are dense and satisfy the corresponding conditional variational principle for topological entropy.

In the notation of the present paper, these conclusions follow from Theorems \ref{thm:A} and \ref{thm:B}. For completeness, we recall the precise statement of the result from \cite{shi2023multifractal}, translated into our notation, in Section \ref{subsec:5.2}, where we also give the formal deduction.

\subsection*{Novelty and main difficulties} Our results establish a flow analogue of the saturated entropy formula known for discrete-time systems with specification or almost product structure \cites{pfister2007topological}, in the setting of singular hyperbolic flows.  Two structural features distinguish the flow setting from the discrete-time one and account for the main difficulties, and a third phenomenon shows that the discrete-time theory cannot be transported directly.

\emph{(i) Passage between the flow and its time-one map.} The measure-theoretic entropy and the Bowen balls are most naturally formulated for the time-one map $\phi_1$, whereas the empirical measures $\mathcal{E}_t(x)$ are continuous-time averages. The bridge is the averaging map
\[
    \tau(\nu)=\int_0^1\nu\circ\phi_{-t}\,dt
\]
of Lemma \ref{lem:tau}, a continuous affine surjection from $\M(\phi_1,\Lambda)$ onto $\M(\Phi,\Lambda)$ that preserves entropy and intertwines the limit sets, $\tau\left(V_{\phi_1}(x)\right) = V_\Phi(x)$. The upper bound then follows from Bowen's estimate (Lemma \ref{lem:Bowen}). However, $\tau$ is neither injective nor a bijection between saturated sets: the example following Lemma \ref{lem:tau} exhibits a set $\mathcal{K}$ with $\phi_{[0,1]}G_{\mathcal{K}}(\phi_1) = \emptyset \subsetneq G_{\tau \mathcal{K}}(\Phi)$. Hence the discrete-time saturated theory cannot simply be transported, and the correspondence must be set up simultaneously at the level of measures, entropy, and supports.

\emph{(ii) Absence of global specification.} Unlike mixing hyperbolic basic sets, geometric Lorenz attractors and singular hyperbolic attractors containing a singularity satisfy neither the weak specification property \cite{sumi2015specification} nor the shadowing property \cite{WenWen2020}. We circumvent this through the horseshoe approximation property \cite{shi2023multifractal}: $\mathcal{K}$ is approximated by ergodic measures supported on a non-decreasing sequence of horseshoes, which are glued together via specification; by additionally interlacing segments of the orbit of a transitive point, the constructed points are made transitive in $\Lambda$ and confined to any prescribed open set, which yields the local form of the formula on $\Trans(\Phi,\Lambda) \cap U$. Nestedness is obtained from the larger-horseshoe construction (Lemma \ref{lem:larger_horseshoe}), and each horseshoe is made topologically mixing---hence enjoys the periodic specification property (Theorem \ref{mixing_implies_specification})---by ensuring that it contains periodic orbits with incommensurate periods, a $C^1$-generic property established in Lemmas \ref{lem:lorenz_residual} and \ref{lem:sectional_residual}. The resulting lower bound is Theorem \ref{thm:5.1}, in which the uniform separation property (Lemma \ref{lem:AEE-to-USP}) and the variational principle for the measure-theoretical lower entropy \cite{feng2012variational} yield the estimate.

\emph{(iii) Failure of generic points for singular flows.} In the proof of Theorems \ref{thm:C} and \ref{thm:D}, near a Lorenz-like singularity or a singular-hyperbolic type singularity $\sigma$, the local geometry implies that, for every $x$, 
either $V_\Phi(x)=\{\delta_\sigma\}$, 
or there exists $\mu\in V_\Phi(x)$ satisfying the fixed mass inequality $\mu(V^*) \geq L\,\mu(U^*)$ between a neighborhood $U^*$ of $\sigma$ and a tubular neighborhood $V^*$ of an associated periodic orbit (Lemma \ref{lem:star}, inequality \eqref{eq:star}). Choosing $\theta$ sufficiently close to $1$, the measure $\mu_\theta = \theta\delta_\sigma + (1-\theta)\mu$, with $h_{\mu}(\phi_{1})>0$, violates this inequality. Thus $G_{\mu_\theta} = \emptyset$ although $h_{\mu_\theta}(\phi_1) > 0$. This obstruction has no analogue in the uniformly hyperbolic case, and it explains why the measure-level approach above is necessary.

\subsection*{Organization} The paper is organized as follows. Section 2 recalls background on hyperbolic sets, horseshoes, geometric Lorenz attractors, singular hyperbolicity, specification and entropy for noncompact sets.
Section 3 proves the upper bound for saturated sets of flows by relating continuous-time empirical measures to the empirical measures of the time-one map.
Section 4 establishes a general lower bound in Theorem \ref{thm:5.1} under asymptotic entropy-expansiveness and specification on approximating subsystems.
Section 5 applies the abstract result to geometric Lorenz attractors and singular hyperbolic attractors, proves Theorems \ref{thm:A} and \ref{thm:B}, and derives the multifractal consequences. 
Finally, Section 6 proves Theorems \ref{thm:C} and \ref{thm:D} by constructing positive-entropy invariant measures without generic points.

\section{Preliminaries}
Throughout this paper, we frequently deal with a flow $\Phi=(\phi_t)_{t\in\mathbb R}$ and its time-one map $\phi_1$ simultaneously. We distinguish the continuous-time and discrete-time settings by including the corresponding flow or map in the notation. For example, we use $\mathcal{E}_{t}(x,\Phi), V_\Phi(x), G_{\mathcal{K}}(\Phi)$ for a flow and $\mathcal{E}_n(x,\phi_1), V_{\phi_{1}}(x), G_{\mathcal{K}}(\phi_{1})$ for its time-one map.

\subsection{Hyperbolic sets, basic sets and horseshoes}
We first describe the uniformly hyperbolic subsystems, namely basic sets and horseshoes.
\begin{definition}[Hyperbolic set for a flow]
Let $M$ be a compact Riemannian manifold, and let $\Phi=(\phi_t)$ be the flow generated by the $C^1$ vector field $X$. A compact $\Phi$-invariant set $\Lambda$ is hyperbolic if there exists a continuous $D\phi_t$-invariant tangent bundle splitting
$$
T_\Lambda M = E^s \oplus \langle X\rangle \oplus E^u
$$
and constants $C\ge 1$ and $\lambda\in(0,1)$ such that for every $x\in \Lambda$ and every $t>0$,
$$
\|D\phi_t|_{E^s_x}\|\le C\lambda^t,\qquad
\|D\phi_{-t}|_{E^u_x}\|\le C\lambda^t.
$$
Here, $\langle X(x)\rangle=\mathbb RX(x)$ is the flow direction if $x$ is a regular point,
and $\langle X(x)\rangle=0$ if $x$ is a singularity.
\end{definition}

\begin{definition}
    Let $X \in \cX^1(M)$ and let $\Phi=(\phi_t)$ be the flow generated by $X$. An invariant compact set $\Lambda$ is called a \emph{basic set} if it is transitive, hyperbolic, and locally maximal, i.e., there exists an open neighborhood U of $\Lambda$, such that $\Lambda = \bigcap_{t \in \Rb}\phi_t(U)$.
\end{definition}

We say the periods of a set of periodic points are \emph{commensurate} if they are all in $p\Zb$ for some $p > 0$. Otherwise, we say they are \emph{incommensurate}, which implies that they generate a dense subgroup of $\Rb$.

\begin{theorem}[\cite{fisher2019hyperbolic}*{Theorem 6.2.11}]\label{incommensurates_implies_mixing}
    Let $\Lambda$ be a locally maximal hyperbolic set containing no singularities of $X$. Then the following statements are equivalent:
    \begin{enumerate}
        \item $\Phi|_{\Lambda}$ is mixing.
        \item The periodic points of $\Lambda$ are dense in $\Lambda$, and for each periodic point $p \in \Lambda$ we have $\overline{W^s(p) \cap \Lambda}=\Lambda$ and $\overline{W^u(p) \cap \Lambda}=\Lambda$.
        \item $\Phi|_{\Lambda}$ is transitive and has periodic points with incommensurate periods.
    \end{enumerate}
\end{theorem}

\begin{definition}
    A basic set $\Lambda$ is called a \emph{horseshoe} if $\Lambda$ is a proper subset of $M$, is not reduced to a single orbit of a hyperbolic critical element, and its intersection with any local cross-section to the flow is totally disconnected.
\end{definition}
\begin{remark*}
    Definitions of horseshoes vary across different literatures. In this paper, we follow the definition given in \cite{shi2023multifractal}*{Definition 4.2}.
\end{remark*}

We give the definition of the homoclinic class of a hyperbolic periodic orbit and the homoclinically related property between hyperbolic periodic orbits. First, recall that for a hyperbolic periodic point $p$ of $X \in \cX^1(M)$, the \emph{stable manifolds} or \emph{unstable manifolds} associated to its orbit $\mathrm{Orb}(p)$ are defined as follows:
\[
    W^s(\mathrm{Orb}(p)) = \left\{ x : \lim_{t \to +\infty} d\left(\phi_t(x), \mathrm{Orb}(p)\right)  = 0 \right\},
\]
\[
    W^u(\mathrm{Orb}(p)) = \left\{ x : \lim_{t \to +\infty} d\left(\phi_{-t}(x), \mathrm{Orb}(p)\right)  = 0 \right\}
\]
where $\Orb(p):=\left\{\phi_t(p):t \in \Rb\right\}$. The classical theory on invariant manifolds by \cite{hirsch1977invariant} ensures that $W^s(\mathrm{Orb}(p))$ and $W^u(\mathrm{Orb}(p))$
are submanifolds of $M$.

\begin{definition}
    Given two distinct hyperbolic periodic points $p, q$ of $X \in \cX^1(M)$, the \emph{homoclinic class} of $p$ is defined as
    \[
        H(p) = \overline{W^s(\mathrm{Orb}(p)) \pitchfork W^u(\mathrm{Orb}(p))},
    \]
    that is, the closure of the set of transversal intersections between the stable and unstable manifolds of the periodic orbit $\mathrm{Orb}(p)$ of $p$. Two hyperbolic periodic orbits $\mathrm{Orb}(p)$ and $\mathrm{Orb}(q)$
    are \emph{homoclinically related} if the stable manifold of $\mathrm{Orb}(p)$ has non-empty transverse intersections with the unstable manifold of $\mathrm{Orb}(q)$ and vice versa. A homoclinic class is called
    \emph{non-trivial} if it is not reduced to a single hyperbolic periodic orbit.
\end{definition}

\begin{lemma}[\cite{palis1982geometric}*{Lemma 7.1}]\label{lem:larger_horseshoe}
    Let $\Lambda_1$ and $\Lambda_2$ be two horseshoes of $X \in \cX^1(M)$. Assume there exist hyperbolic periodic points $p_1 \in \Lambda_1$ and $p_2 \in \Lambda_2$ such that $\Orb\left(p_1\right) $ and $\Orb\left(p_2\right) $ are homoclinically related. Then there exists a larger horseshoe $\Lambda$ that contains both $\Lambda_1$ and $\Lambda_2$.
\end{lemma}

Lemma \ref{lem:larger_horseshoe} shows that, homoclinic relatedness allows different horseshoes to be combined into a larger one, while the following approximation property allows invariant measures to be approximated by ergodic measures supported on horseshoes without a significant loss of entropy. 

Here and below, any probability measure supported on a compact invariant subset of $\Lambda$ is naturally regarded as an element of $\M(\Lambda)$. We denote by $D$ the first Wasserstein distance on $\M(\Lambda)$ induced by the ambient metric on $M$; it is also called the Kantorovich--Rubinstein distance, see \cite{villani2009optimal}*{Definition 6.1}. All distances between such measures are computed in the ambient space $\M(\Lambda)$.

\begin{definition}
    Let $X \in \cX^1(M)$ and $\Lambda$ be a $\Phi$-invariant compact set. We say a convex subset $\cK \subseteq \M(\Phi, \Lambda)$ has the \emph{horseshoe approximation property} if for any $\ve > 0$ and any $\mu \in \cK$, there exists a horseshoe $\Lambda' \subset \Lambda$ and $\nu \in \Merg(\Phi, \Lambda')$ such that $D(\mu, \nu) < \ve$ and $h_\nu\left(\phi_1\right)  > h_\mu\left(\phi_1\right) - \ve$.
\end{definition}

In this article, we say $\Lambda$ has the horseshoe approximation property if $\M(\Phi, \Lambda)$ has the horseshoe approximation property.

The horseshoe approximation property will later allow us to replace the singular hyperbolic dynamics by a non-decreasing sequence of mixing horseshoes. We next recall the geometric classes of attractors in which this
approximation property is available generically.

\subsection{Geometric Lorenz attractor and singular hyperbolicity}
For a vector field $X \in \cX^1(M)$, denote its set of singularities by
\[
    \Sing(X):=\{x\in M:X(x)=0\}.
\]
For a compact $\Phi$-invariant set $\Lambda$, write
\[
    \Sing(\Lambda):=\Sing(X)\cap\Lambda.
\]
\subsubsection{Singular hyperbolicity}
The concept of singular hyperbolicity was introduced by Morales, Pacifico and Pujals \cite{morales2004robust} to describe the geometric structure of Lorenz attractors and these ideas were extended to higher dimensional cases in \cites{li2005robustly, metzger2008sectional}.

\begin{definition}
    Given a flow $\Phi=\{ \phi_{t} \}$ generated by a vector field $X \in \cX^1(M)$, a compact $\phi_{t}$-invariant set $\Lambda$ is \emph{singular hyperbolic} if it admits a continuous $D\phi_t$-invariant splitting $T_\Lambda M = E^{ss} \oplus E^{cu}$ and constants $C > 1$ and $\eta > 0$, such that for any $x \in \Lambda$ and any $t \geq 0$,
    \begin{enumerate}
        \item $E^{ss}$ is dominated by $E^{cu}$: $\left\lVert D\phi_t|_{E^{ss}_x}\right\rVert \cdot \left\lVert D\phi_{-t}|_{E^{cu}_{\phi_t(x)} }\right\rVert < Ce^{-\eta t}$
        \item $E^{ss}$ is uniformly contracted: $\left\lVert D\phi_{t}(v) \right\rVert < Ce^{-\eta t} \left\lVert v\right\rVert $ for any $v \in E^{ss}_x \backslash \left\{0\right\} $.
        \item $E^{cu}$ is sectionally expanded: $\left\lvert \det D\phi_t|_{V_x}\right\rvert > C^{-1}e^{\eta t}$ for any two-dimensional subspace $V_x \subseteq E^{cu}_x$.
    \end{enumerate}
\end{definition}

\begin{proposition}\label{prop:sec-hyp}
    Let $\Phi=\{ \phi_{t} \}$ be a flow generated by a vector field $X \in \cX^1(M)$, and let $\Lambda$ be a compact $\phi_t$-invariant transitive set. Then, the following properties hold:
    \begin{enumerate}
        \item All hyperbolic periodic orbits of a singular hyperbolic set $\Lambda$ have the same index.
        \item If $\Lambda$ is hyperbolic and contains regular points, then $\Sing(\Lambda) = \emptyset$.
        \item \cite{morales2004robust}*{Proposition 1.8} If $\Sing(\Lambda) = \emptyset$, then $\Lambda$ is hyperbolic if and only if $\Lambda$ is singular hyperbolic for $X$ or $-X$.
        \item \cite{morales1999singular}*{Proposition 1} singular hyperbolicity is $C^1$-robust.
    \end{enumerate}
\end{proposition}

\subsubsection{Geometric Lorenz attractors}
Next, we give the definition of geometric Lorenz attractors following Guckenheimer and Williams \cites{guckenheimer1979structural,williams1979structure} for vector fields on a closed $3$-manifold $M^3$. For a hyperbolic singularity $\sigma$ of $X$, its local stable manifold of size $\delta$ is
\[W_\delta^s(\sigma) = \left\{x \in M: \lim_{t \to +\infty} \phi_t(x) = \sigma \mathrm{\ and\ } d\left(\phi_t(x), \sigma\right) < \delta \mathrm{\ for\ any\ } t \geq 0 \right\}.\]
An invariant compact set $\Lambda$ is called an \emph{attracting set} if there exists a neighborhood $U$ of $\Lambda$ such that $\phi_t(\overline{U}) \subset U $ for any $t > 0$ and $\bigcap_{t > 0}\phi_t(\overline{U}) = \Lambda$. Such a neighborhood $U$ will be called an \emph{attracting
neighborhood} of $\Lambda$. $\Lambda$ is called an \emph{attractor} if it is a transitive attracting set.

\begin{definition}
    We say a singular hyperbolic attractor $\Lambda$ of $X \in \cX^r\left(M^3\right)$, where $r \geq 1$, is a \emph{geometric Lorenz attractor} if there exist an attracting neighborhood $U$ of $\Lambda$
and a $C^r$-smooth cross-section $\Sigma\subset U$ such that:
    \begin{itemize}
        \item There exists a unique singularity $\sigma \in \Lambda$ with three exponents $\lambda_1 < \lambda_2 < 0 < \lambda_3$, which satisfy $\lambda_1 + \lambda_3 < 0$ and $\lambda_2 + \lambda_3 > 0$.
        \item $\Sigma$ is $C^1$-diffeomorphic to $[-1, 1] \times [-1, 1]$, and $l = \{0\} \times [-1, 1] = W^s_{\text{loc}}(\sigma) \cap \Sigma$. Moreover, for every $z \in U \setminus W^s_{\text{loc}}(\sigma)$, there exists $t > 0$ such that $\phi_t^X(z) \in \Sigma$.
        \item Up to the previous identification, the Poincaré map $P : \Sigma \setminus l \to \Sigma$ is a skew-product map
              \begin{equation*}
                  P(x, y) = (f(x), H(x, y)), \quad \forall (x, y) \in [-1, 1]^2 \setminus l.
              \end{equation*}
              Moreover, it satisfies
              \begin{itemize}
                  \item $H(x, y) < 0$ for $x > 0$, and $H(x, y) > 0$ for $x < 0$;
                  \item $\sup_{(x, y) \in \Sigma \setminus l} |\partial H(x, y) / \partial y| < 1$, and $\sup_{(x, y) \in \Sigma \setminus l} |\partial H(x, y) / \partial x| < 1$;
                  \item the one-dimensional quotient map $f : [-1, 1] \setminus \{0\} \to [-1, 1]$ is $C^1$-smooth and satisfies $\lim_{x \to 0^-} f(x) = 1$, $\lim_{x \to 0^+} f(x) = -1$, $-1 < f(x) < 1$ and $f'(x) > \sqrt{2}$ for every $x \in [-1, 1] \setminus \{0\}$.
              \end{itemize}
    \end{itemize}
\end{definition}

The lower-bound argument will be carried out on a sequence of mixing horseshoes contained in $\Lambda$, rather than directly on the ambient singular hyperbolic attractor. We therefore recall the specification property and the fact that every mixing horseshoe, being a compact locally maximal hyperbolic set, has the periodic specification property.

\subsection{Specification property}

\begin{definition}[see \cite{fisher2019hyperbolic}*{Definition 7.3.2}] \label{def:specification}
    A flow $\Phi$ on a compact metric space $K$ satisfies the \emph{specification property} if, for any $\ve > 0$, there is a $T_{\ve}>0$ such that for any $x_0, \dots, x_n \in K$ and any $t_0, \dots, t_{n+1} \in [0, \infty)$ with $t_{i+1} \geq t_i + T_{\ve}$, there is $y \in K$ with
    \[d\left(\phi_t(y), \phi_{t-t_i}(x_i)\right)  < \ve \quad \text{for\ any\ } t \in \left[t_i, t_{i+1} - T_{\ve}\right].\]
    We also say that $y$ $\ve$-shadows the orbit segments and denote by
    \[\overset{t_0}{\longrightarrow}\left(x_0,t_1-t_0-T_\ve\right)
        \overset{T_\ve}{\longrightarrow}\left(x_1,t_2-t_1-T_\ve\right)
        \overset{T_\ve}{\longrightarrow}\dots
        \overset{T_\ve}{\longrightarrow}\left(x_n,t_{n+1}-t_n-T_\ve\right).\]
    If further $y$ can be chosen to be a periodic point of $\Phi$ with period in $[t_{n+1}-t_{0}-\varepsilon,t_{n+1}-t_{0}+\varepsilon]$, then we say $\Phi$ satisfies the \emph{periodic specification property}.
\end{definition}

\begin{theorem}[see \cite{fisher2019hyperbolic}*{Theorem 7.3.4}] \label{mixing_implies_specification}
    If $\Lambda$ is a mixing compact locally maximal hyperbolic set for $\Phi$, then $\Phi|_{\Lambda}$ has the periodic specification property.
\end{theorem}
We next recall the entropy notions needed to formulate the upper and lower bounds for the noncompact saturated sets.

\subsection{Entropy for noncompact sets}
In this subsection, we recall the definition of topological entropy for noncompact sets.
\subsubsection{Entropy for discrete-time continuous maps}
The original definition of topological entropy for noncompact sets was introduced by Bowen \cite{bowen1973topological} for discrete-time continuous maps. Let $f:K \to K$ be a continuous map on a compact metric space $(K,d)$.

For any $x \in K$, $n \in \Nb$ and any $\ve > 0$, define
\[B_n(x,\ve,f) = \left\{y \in K: d\left(f^ix, f^iy\right) < \ve \mathrm{\ for\ any\ } i \in \left\{0,1,\dots,n-1\right\} \right\}.\]
\begin{definition}\label{def:htop}
    Let $E \subseteq K$, $\ve > 0$, and $\fF_n(E,\ve)$ be the collection of all finite or countable covers of $E$ by sets of the form $B_m(x,\ve,f)$ with $m \geq n$.
    $$C(E;s,n,\ve
        ,f) = \inf_{\mathcal{C} \in \fF_n(E,\ve)}\left\{\sum_{B_m(x,\ve,f) \in \mathcal{C}} e^{-sm}\right\}$$
    and
    $$C(E;s,\ve,f) = \lim_{n\to\infty}C(E;s,n,\ve,f).$$
    Let
    $$\htop(f,E,\ve)=\inf\left\{s:C(E;s,\ve,f) = 0\right\} = \sup\left\{s:C(E;s,\ve,f) = \infty\right\}.$$
    The \emph{(Bowen) topological entropy} of $E$ with respect to $f$ is defined by
    $$\htop(f,E) = \lim_{\ve\to 0}\htop(f,E,\ve).$$
\end{definition}

Next, we recall the definition of entropy-expansive maps. For any $x \in K$ and any $\ve > 0$, let $B_{\infty}(x,\ve,f) = \left\{y \in K: d\left(f^n x,f^n y\right) \leq \ve \text{ for any } n \in \Nb \right\}$.

\begin{definition}
    We say $\left(K,f\right)$ is entropy-expansive if there exists $\ve > 0$ such that
    \[\sup_{x \in K} \htop\left(f,B_{\infty}(x,\ve,f)\right) = 0.\]
    We say $\left(K,f\right)$ is asymptotically entropy-expansive if
    \[\lim_{\ve \to 0} \sup_{x \in K} \htop\left(f,B_{\infty}(x,\ve,f)\right) = 0.\]
\end{definition}
\begin{remark*}
    When $f$ is a homeomorphism, one may replace $B_{\infty}(x,\ve,f)$ by
    \[B^{\pm}_{\infty}(x,\ve,f) := \left\{y \in K: d\left(f^n x,f^n y\right) \leq \ve \text{ for any } n \in \Zb\right\}.\]
    Indeed, Bowen \cite{bowen1972entropy}*{Corollary 2.3} proved that
    \[\sup_{x \in K} \htop\left(f,B_{\infty}(x,\ve,f)\right) = \sup_{x \in K} \htop\left(f,B^{\pm}_{\infty}(x,\ve,f)\right)\]
    for every $\ve > 0$.
\end{remark*}

When there is no ambiguity, we also say $f$ is entropy-expansive or asymptotically entropy-expansive and omit the space $K$.

\subsubsection{Entropy for flows}
In this subsection, we recall the definition of topological entropy for noncompact sets for flows. Readers may refer to \cite{pacifico2022bowen}*{Section 3.2} for details.

Let $\Phi = \left(\phi_t\right)_{t \in \Rb}$ be a continuous flow on a compact metric space $K$. For any $x \in K$, $t \in \Rb^+$ and any $\ve > 0$, define
\[B_t(x,\ve,\Phi) = \left\{y \in K: d\left(\phi_s x, \phi_s y\right) < \ve \mathrm{\ for\ any\ } s \in [0,t] \right\}.\]

Let \(\mathcal{U}\) be a finite open cover of \(K\). If \(B \subseteq K\), we say \(B \prec \mathcal{U}\) if \(B\) is contained in some element of \(\mathcal{U}\). We say a collection \(\{B_i:i \in I\} \prec \mathcal{U}\) if each \(B_i\) is contained in some element of \(\mathcal{U}\).

Now, we recall some notations and definitions.
\begin{enumerate}[label=(\roman*)]
    \item \(N(B) = N(\Phi, \mathcal{U}, B)\) is defined as the largest non-negative number such that \(\phi_t(B) \prec \mathcal{U}\) for any \(0 \leq t < N(B)\). We define \(N(B) = 0\) if \(B \not\prec \mathcal{U}\) and \(N(B) = \infty\) if \(\phi_t(B) \prec \mathcal{U}\) for any \(t \geq 0\).

    \item \(\mathsf{d}(B) := \mathsf{d}(\Phi, \mathcal{U}, B) =
          \begin{cases}
              e^{-N(B)} & \text{if } N(B) \text{ is finite}, \\
              0         & \text{if } N(B) = \infty.
          \end{cases}\)

    \item \(S_{\alpha}(B) := S_{\alpha}(\Phi, \mathcal{U}, B) = \sum_{i \in \Nb^+} \mathsf{d}\left(B_i\right) ^{\alpha}\) if \(\alpha \in \mathbb{R}\) and \(B = \left\{B_i:i \in \Nb^+\right\} \) is a countable collection of subsets of \(K\).
\end{enumerate}

For any \(Y \subseteq K\), define
\[\Gamma_{\alpha}(Y) := \Gamma_{\alpha}(\Phi, \mathcal{U}, Y) = \lim_{\varepsilon \to 0} \inf\left\{S_{\alpha}(\xi) : \xi = \left\{E_i\right\} _{i \geq 1} \text{ is a cover of } Y \text{ with } \mathsf{d}\left(E_i\right) < \varepsilon\right\}.\]
Note that \(\Gamma_{\alpha}\) is an exterior measure and satisfies \(\Gamma_{\alpha}(Y) \leq \Gamma_{\beta}(Y)\) if \(\alpha > \beta\). Moreover, \(\Gamma_{\alpha}(Y) \notin \{0, \infty\}\) for at most one unique value \(\alpha\). Next, define
\[
    h_{\mathcal{U}}(\Phi, Y) := \inf\left\{\alpha : \Gamma_{\alpha}(Y) = 0\right\}.
\]
The \emph{topological entropy of \(\Phi\) over \(Y\)} is defined by
\[
    \htop(\Phi, Y) := \sup\left\{h_{\mathcal{U}}(\Phi, Y) : \mathcal{U} \text{ is a finite open cover of } K \right\}.
\]

Here, we list some basic facts about the topological entropy.

\begin{lemma}[see \cite{pacifico2022bowen}*{Theorem A, Proposition 3.6}]\label{lem:htop}
    Let $\Phi = \left(\phi_t\right)_{t \in \Rb}$ be a continuous flow on a compact metric space $K$. Then the following statements hold for any $Y, Z \subseteq K$.
    \begin{enumerate}[label=(\roman*)]
        \item \(\htop(\Phi,Y) = \htop\left(\phi_1, Y\right)\).
        \item If \(Y \subseteq Z\), then \(\htop(\Phi, Y) \leq \htop(\Phi, Z)\).
        \item If \(Y = \bigcup_{i=1}^\infty Y_i\), then \(\htop\left(\Phi, Y\right) = \sup_{1 \leq i \leq \infty}\htop\left(\Phi, Y_i\right)\).
        \item \(\htop(\Phi,\emptyset) = 0\).
    \end{enumerate}
\end{lemma}

\begin{definition}
    We say $\left(K,\Phi\right)$ is entropy-expansive (resp. asymptotically entropy-expansive) if its time-one map $\phi_1$ is entropy-expansive (resp. asymptotically entropy-expansive). 
\end{definition}

By \cite{bowen1972entropy}*{Example 1.6}, the time-$t$ map of an expansive flow (see \cite{bowen1972expansive}*{Definition 1} for definitions) is entropy-expansive. Therefore, if $\left(K,\Phi\right)$ is expansive, then $\left(K,\Phi\right)$ is entropy-expansive. If $\left(K,\Phi\right)$ is entropy-expansive, then $\left(K,\Phi\right)$ is asymptotically entropy-expansive.

\section{Upper bounds of topological entropy for saturated sets} \label{sec:3}
The purpose of this section is to establish the upper bound for the topological entropy of saturated sets of flows. The main point is to compare continuous-time empirical measures with empirical measures of the time-one map, and then apply Bowen's entropy estimate for quasi-regular sets.

Throughout this section, let $\Phi=(\phi_t)_{t\in\mathbb R}$ be a continuous flow on a compact metric space $K$.  We use the $1$-Wasserstein metric $D$ on $\mathcal M(K)$ induced by the metric of $K$. Here and below, whenever Wasserstein estimates are used, we normalize the ambient metric so that its diameter is at most $1$. 

\begin{theorem}\label{thm:upper_htop}
     For any nonempty compact connected $\mathcal{K} \subseteq \M(\Phi,K)$, we have
    \[\htop\left(\Phi, G_{\mathcal{K}}(\Phi)\right) \leq \inf\left\{h_\mu\left(\phi_1\right):\mu \in \mathcal{K} \right\}.\]
\end{theorem}
This upper bound is completely
general: it does not use hyperbolicity, expansiveness or specification. The rest of the section is devoted to the proof of this estimate.

\subsection{Reduction to the time-one map}
For $a\geq0$, define the \emph{quasi-regular set} by
\[
QR(a)=
\{x\in K:\ h_\mu(\phi_1)\leq a
\text{ for some } \mu\in V_{\phi_1}(x)\}.
\]
This set records points whose discrete-time empirical measures admit an
accumulation measure with entropy at most $a$.  The following lemma is the
only step in the upper-bound argument where the continuous-time saturated
condition is converted into such a discrete-time condition.
\begin{lemma}\label{lem:saturated-contained-qr}
    For any nonempty compact connected $\mathcal{K} \subseteq \M(\Phi,K)$, we have
    \[G_\mathcal{K}(\Phi) \subseteq QR\left(\inf\left\{h_\mu\left(\phi_1\right):\mu \in \mathcal{K} \right\}  + \ve\right)\]
    for any $\ve > 0$.
\end{lemma}
Once this inclusion is available, the desired entropy bound follows from the
following estimate for the sets $QR(t)$.  
\begin{lemma}[\cite{bowen1973topological}*{Theorem 2}]\label{lem:Bowen}
    For any $t \geq 0$, we have
    \[\htop\left(\phi_1,QR(t)\right) \leq t.\]
\end{lemma}
Combining these two lemmas with the equality
$h_{\operatorname{top}}(\Phi,Y)=h_{\operatorname{top}}(\phi_1,Y)$ for $Y\subset K$ gives the announced upper bound.

\begin{proof}[Proof of Theorem~\ref{thm:upper_htop}]
If $\inf_{\mu\in\mathcal K}h_\mu(\phi_1)=+\infty$, there is nothing to
prove.  Otherwise, fix $\varepsilon>0$.  By Lemma~\ref{lem:htop},
\[
h_{\operatorname{top}}\bigl(\Phi,G_{\mathcal K}(\Phi)\bigr)
=
h_{\operatorname{top}}\bigl(\phi_1,G_{\mathcal K}(\Phi)\bigr).
\]
By Lemma~\ref{lem:saturated-contained-qr} and the monotonicity of topological
entropy for noncompact sets,
\[
h_{\operatorname{top}}\bigl(\phi_1,G_{\mathcal K}(\Phi)\bigr)
\leq
h_{\operatorname{top}}\left(
\phi_1,
QR\left(\inf_{\mu\in\mathcal K}h_\mu(\phi_1)+\varepsilon\right)
\right).
\]
Applying Lemma~\ref{lem:Bowen}, we obtain
\[
h_{\operatorname{top}}\bigl(\Phi,G_{\mathcal K}(\Phi)\bigr)
\leq
\inf_{\mu\in\mathcal K}h_\mu(\phi_1)+\varepsilon.
\]
Since $\varepsilon>0$ is arbitrary, the proof is complete.
\end{proof}
It remains to prove Lemma~\ref{lem:saturated-contained-qr}.  We now prepare the two ingredients needed to compare the continuous-time and discrete-time empirical behaviors.  The first lemma controls the change of empirical measures after removing a small part of a time interval or replacing an orbit segment by a nearby one.  The second lemma studies the averaging map which sends $\phi_1$-invariant measures to $\Phi$-invariant measures.

\subsection{Proofs of the auxiliary lemmas}
\subsubsection{Comparison of empirical measures}
We begin with an elementary estimate for empirical measures.  It says that discarding a small proportion of a time interval changes the empirical measure by a small amount in the 1-Wasserstein metric, and that two close orbit pieces have close empirical measures.

We shall use the following elementary properties of the 1-Wasserstein metric $D$, which follow directly from the definition, see \cite{villani2009optimal}*{Section 6}.  
\begin{itemize}
    \item for any $x,y\in K$, $D(\delta_x,\delta_y)=d(x,y)$.
    \item under the assumption that $\operatorname{diam} K\leq 1$,  $D(\eta,\xi)\le 1$ for all $\eta,\xi\in\mathcal M(K)$.
    \item $D$ is convex in the following sense: if $\mu_i,\nu_i\in\mathcal M(K)$ and $a_i\geq0$ with $\sum_{i=1}^m a_i=1$, then
$$
D\left(\sum_{i=1}^m a_i\mu_i,\sum_{i=1}^m a_i\nu_i\right)
\leq
\sum_{i=1}^m a_iD(\mu_i,\nu_i).
$$
\end{itemize}

\begin{lemma}\label{lem:D}
    Let $\ve > 0$, $0<\delta<1$,$t>0$. Let $\sigma_1:[0,t] \to K$ and $\sigma_2:[0,t] \to K$ be two measurable maps. For any measurable set $I \subseteq [0,t]$ with $\mathrm{Leb}(I) > (1-\delta)t$, one has
    \[D\left(\frac{1}{t}\int_{0}^{t}\delta_{\sigma_1(s)}\dif s, \frac{1}{\mathrm{Leb}(I)}\int_{I}\delta_{\sigma_1(s)}\dif s\right) < \delta.\]
    If further $d\left(\sigma_1(s),\sigma_2(s)\right)  < \ve$ for any $s \in I$, then
    \[D\left(\frac{1}{t}\int_{0}^{t}\delta_{\sigma_1(s)}\dif s, \frac{1}{t}\int_{0}^{t}\delta_{\sigma_2(s)}\dif s\right) < \ve + \delta\]
    and hence
    \[D\left(\frac{1}{t}\int_{0}^{t}\delta_{\sigma_1(s)}\dif s, \frac{1}{\mathrm{Leb}(I)}\int_{I}\delta_{\sigma_2(s)}\dif s\right) < \ve + 2\delta.\]
    In particular, if $d\left(\sigma_1(s),\sigma_2(s)\right)  < \ve$ holds for any $s \in [0,t]$, then
    \[D\left(\frac{1}{t}\int_{0}^{t}\delta_{\sigma_1(s)}\dif s, \frac{1}{t}\int_{0}^{t}\delta_{\sigma_2(s)}\dif s\right) \leq \ve.\]
\end{lemma}
\begin{proof}
Set
\[
A=[0,t]\setminus I,\qquad \alpha=\frac{\operatorname{Leb}(A)}{t},\qquad
\beta=\frac{\operatorname{Leb}(I)}{t},
\]
so that $\alpha+\beta=1$ and $\alpha<\delta$.
For $j=1,2$, define
\[
\mu_A^j=\frac{1}{\operatorname{Leb}(A)}\int_A \delta_{\sigma_j(s)}\,ds
\quad\text{(if }\operatorname{Leb}(A)>0\text{)},
\qquad
\mu_I^j=\frac{1}{\operatorname{Leb}(I)}\int_I \delta_{\sigma_j(s)}\,ds.
\]
Then
\[
\frac1t\int_0^t \delta_{\sigma_j(s)}\,ds
=\alpha\,\mu_A^j+\beta\,\mu_I^j.
\]
By the convexity of $D$ and $D(\eta,\xi)\le 1$ for all
$\eta,\xi\in\mathcal M(K)$, we first have
\[
D\!\left(\frac1t\int_0^t \delta_{\sigma_1(s)}\,ds,\mu_I^1\right)
= D(\alpha\mu_A^1+\beta\mu_I^1,\mu_I^1)
\le \alpha D(\mu_A^1,\mu_I^1)
\le \alpha<\delta.
\]
Next, assume $d(\sigma_1(s),\sigma_2(s))<\varepsilon$ for all $s\in I$. Then
by convexity and $D(\delta_x,\delta_y)=d(x,y)$,
\[
D(\mu_I^1,\mu_I^2)
\le \frac{1}{\operatorname{Leb}(I)}\int_I
D(\delta_{\sigma_1(s)},\delta_{\sigma_2(s)})\,ds
= \frac{1}{\operatorname{Leb}(I)}\int_I
d(\sigma_1(s),\sigma_2(s))\,ds
<\varepsilon.
\]
Moreover, 
\[
D\!\left(\frac1t\int_0^t \delta_{\sigma_1(s)}\,ds,
\frac1t\int_0^t \delta_{\sigma_2(s)}\,ds\right)
\le \alpha D(\mu_A^1,\mu_A^2)+\beta D(\mu_I^1,\mu_I^2)
< \alpha+\varepsilon
< \delta+\varepsilon.
\]
Applying the first estimate to $\sigma_2$ gives
\[
D\!\left(\frac1t\int_0^t \delta_{\sigma_2(s)}\,ds,\mu_I^2\right)<\delta.
\]
Therefore, by the triangle inequality,
\[
D\!\left(\frac1t\int_0^t \delta_{\sigma_1(s)}\,ds,\mu_I^2\right)
< \varepsilon+2\delta.
\]

Finally, if $d(\sigma_1(s),\sigma_2(s))<\varepsilon$ for all $s\in[0,t]$,
then taking $I=[0,t]$ yields $\alpha=0$, and thus
\[
D\!\left(\frac1t\int_0^t \delta_{\sigma_1(s)}\,ds,
\frac1t\int_0^t \delta_{\sigma_2(s)}\,ds\right)
\le \frac1t\int_0^t d(\sigma_1(s),\sigma_2(s))\,ds
< \varepsilon.\]
\end{proof}
\begin{remark}\label{rmk:D}
    The above lemma is the flow-version of the following lemma.
    Let $\ve > 0$, $\delta > 0$ and $\left\{x_i\right\}_{i=0}^{n-1}$, $\left\{y_i\right\}_{i=0}^{n-1}$ be two sequences in $K$. For any nonempty set $J \subseteq \left\{0,1,\dots,n-1\right\}$ with $\left(n-\left\lvert J\right\rvert \right) / n < \delta $,
    \[D\left(\frac{1}{n}\sum_{i=0}^{n-1}\delta_{x_i}, \frac{1}{\left\lvert J\right\rvert }\sum_{i \in J}\delta_{x_i}\right) < \delta.\]
    If further $d(x_i, y_i) < \ve$ holds for any $i \in J$, then
    \[D\left(\frac{1}{n}\sum_{i=0}^{n-1}\delta_{x_i}, \frac{1}{n}\sum_{i=0}^{n-1}\delta_{y_i}\right) < \ve + \delta\]
    and hence
    \[D\left(\frac{1}{n}\sum_{i=0}^{n-1}\delta_{x_i}, \frac{1}{\left\lvert J\right\rvert }\sum_{i \in J}\delta_{y_i}\right) < \ve + 2\delta.\]
    In particular, if $d(x_i, y_i) \leq \ve$ holds for any $i \in \left\{0,1,\dots,n-1\right\}$, then
    \[D\left(\frac{1}{n}\sum_{i=0}^{n-1}\delta_{x_i}, \frac{1}{n}\sum_{i=0}^{n-1}\delta_{y_i}\right) \leq \ve.\]
These estimates follow from the same argument as in Lemma~\ref{lem:D}, with normalized counting measure on ${0,\ldots,n-1}$ in place of normalized Lebesgue measure on $[0,t]$; we therefore omit the proof.

\end{remark}
\subsubsection{The averaging map}
We define the averaging map $\tau:\mathcal{M}(K) \rightarrow \mathcal{M}(K)$ as
$$ \tau(\nu)=\int_0^1 \nu\circ\phi_{-t}\,dt. $$
A direct calculation shows that the averaging map relates empirical measures of the time-one map to empirical measures of the flow:
\begin{equation}\label{eq:phi_to_Phi}
    \tau(\mathcal{E}_{n}(x,\phi_{1}))=\int_{0}^{1}\mathcal{E}_{n}(x,\phi_{1}) \circ \phi_{-t}\dif t=\int_{0}^{1}\mathcal{E}_{n}(\phi_{t}x,\phi_{1}) \dif t=\mathcal{E}_{n}(x,\Phi).
\end{equation}

In particular, Lemma~\ref{lem:D} is used to show that, when $t$ is large, the continuous-time empirical measure $\mathcal{E}_t(x,\Phi)$ can be approximated by the averaged discrete-time empirical measure $\tau \mathcal{E}_{\lfloor t\rfloor}(x,\phi_{1})$, where $\left\lfloor t\right\rfloor$ is the greatest integer not exceeding $t$. That is,
\begin{equation} \label{eq:phi_to_Phi2}
    D\left(\cE_{t}(x,\Phi),\tau\cE_{\left\lfloor t\right\rfloor}\left(x,\phi_{1}\right)\right)= D\left(\cE_{t}(x,\Phi),\cE_{\left\lfloor t\right\rfloor}\left(x,\Phi\right)\right)< 1 / t.
\end{equation}

For a set $\mathcal L\subset\mathcal M(\phi_1,K)$, write
\[
\tau\mathcal L=\{\tau\nu:\nu\in\mathcal L\}.
\]
For a set $A\subset K$, write
\[
\phi_{[0,1]}A=\{\phi_t(x):x\in A,\ 0\leq t\leq1\}.
\]
The following lemma collects the basic properties of $\tau$ needed to prove Lemma~\ref{lem:saturated-contained-qr}.

\begin{lemma}\label{lem:tau}
    The averaging map $\tau:\mathcal M(K)\to\mathcal M(K)$ defined above has the following properties.
    \begin{enumerate}
        \item \label{item:4.1.1} $\tau$ is continuous and affine.
        \item \label{item:4.1.3} $\tau\M\left(\phi_1, K\right) = \M\left(\Phi, K\right)$.
        \item \label{item:4.1.4} $\tau\Merg\left(\phi_1, K\right) = \Merg\left(\Phi, K\right)$.
        \item \label{item:4.1.5} $\phi_{[0,1]}\supp(\mu) = \supp(\tau\mu)$ for any $\mu \in \M(K)$. In particular, $\supp(\mu) \subseteq \supp(\tau\mu)$.
        \item \label{item:4.1.6} $h_{\tau\mu}\left(\phi_1\right) = h_{\mu}(\phi_1)$ for any $\mu \in \M\left(\phi_1, K\right)$.
        \item \label{item:4.1.7} $\tau\left(V_{\phi_1}(x)\right)=V_{\Phi}(x)$ for any $x \in K$.
        \item \label{item:4.1.8} $\phi_{[0,1]}G_{\mathcal{K}}\left(\phi_1\right) \subseteq G_{\tau \mathcal{K}}(\Phi)$ for any nonempty compact connected $\mathcal{K} \subseteq \M(\phi_1,K)$. In particular, $G_\mathcal{K}\left(\phi_1\right) \subseteq G_{\tau \mathcal{K}}(\Phi)$.
    \end{enumerate}
\end{lemma}
\begin{proof}
    The affinity of $\tau$ is immediate from the linearity of the integral. To prove $\tau$ is continuous, it suffices to prove that for any sequence $\left\{\mu_n:n\in\Nb^+\right\} \subseteq \M(K)$ with $\mu = \lim_{n \to \infty} \mu_n$ (in weak* topology), one has $\tau(\mu) = \lim_{n \to \infty} \tau(\mu_n)$. Fix $\left\{\mu_n:n\in\Nb^+\right\} \subseteq \M(K)$ with $\mu = \lim_{n \to \infty} \mu_n$. For any $g \in C(K)$, define
    \[G_n(t) = \int_x g\left(\phi_t(x)\right) \dif \mu_n(x)\]
    and
    \[G(t) = \int_x g\left(\phi_t(x)\right) \dif \mu(x).\]
    For any $t \in [0,1]$, the map $x \to g\left(\phi_t(x)\right)$ is continuous on $K$, which implies that $\left\{G_n(t):n \in \Nb^+\right\}$ converges pointwise to $G(t)$ for $t \in [0,1]$. Since $K$ is compact, there exists $M > 0$ such that $\left\lvert G_n(t)\right\rvert < M$ for any $t \in [0,1]$. By the dominated convergence Theorem,
    \[\lim_{n \to \infty} \int_0^1 G_n(t) \dif t = \int_{0}^{1} G(t) \dif t,\]
    which means
    \[\lim_{n \to \infty} \int_K g \dif\left(\tau\mu_n\right) = \int_K g \dif (\tau\mu).\]
    Since $g$ is arbitrary, $\tau$ is continuous. Then item (\ref{item:4.1.1}) is proved.

    Next, we prove $\Merg(\Phi, K) \subseteq \tau\left(\M\left(\phi_1,K\right) \right)$. For any $\mu \in \Merg(\Phi, K)$, we have $\mu\left(G_\mu(\Phi)\right) = 1$ by the Birkhoff Ergodic Theorem. Take $x \in G_\mu(\Phi)$, then
    \begin{equation}\label{eq:4.1}
        \lim_{t \to +\infty} \cE_{t}(x,\Phi) = \mu.
    \end{equation}
    Take $\nu \in V_{\phi_1}(x)$, then there exists a subsequence $\left\{\cE_{n_k}(x,\phi_{1}) :k \in \Nb^+\right\}$ where $\left\{n_k:k \in \Nb^+\right\} \subseteq \Nb^+$, such that
    \[\lim_{k \to \infty} \cE_{n_k}(x,\phi_{1}) = \nu.\]
    Since $\tau$ is continuous, by (\ref{eq:phi_to_Phi}) and (\ref{eq:4.1}) we have
     $$   \tau\nu = \lim_{k \to \infty} \tau\left(\cE_{n_k}(x,\phi_{1})\right)  = \lim_{k \to \infty} \cE_{n_k}(x,\Phi)= \mu. $$
    Since $\mu$ is arbitrary, we have $\Merg(\Phi, K) \subseteq \tau\left(\M\left(\phi_1,K\right) \right)$. Since $\tau$ is affine, by the ergodic decomposition theorem, we have $\M(\Phi, K) \subseteq \tau\left(\M\left(\phi_1,K\right) \right)$.

    For any $\mu \in \Merg\left(\phi_1, K\right)$ and any measurable $A \subseteq K$ with $\phi_{-t} A = A$ for any $t \in \Rb$, we have
    \[\left(\tau\mu\right)(A) = \int_{0}^{1}\mu\left(\phi_{-t}A\right)\dif t = \mu(A),\]
    which is $0$ or $1$. Therefore, $\tau\Merg\left(\phi_1, K\right) \subseteq \Merg\left(\Phi, K\right)$. By the Ergodic Decomposition theorem and item (\ref{item:4.1.1}), we have $\tau\left(\M\left(\phi_1,K\right) \right) \subseteq \M(\Phi, K)$. Then item (\ref{item:4.1.3}) is proved.

    Next, we prove item (\ref{item:4.1.4}). It remains to prove  $\tau^{-1}(\nu) \cap \Merg\left(\phi_1, K\right) \ne \emptyset$ for any $\nu \in \Merg\left(\Phi, K\right)$. Fix $\nu \in \Merg\left(\Phi,K\right)$. By item \eqref{item:4.1.3}, there exists $\mu \in \M\left(\phi_1, K\right)$ such that $\tau\mu = \nu$. Suppose $\mu$ has an ergodic decomposition $\mu = \int \mu_\xi \dif \kappa(\xi)$. By $\tau\Merg\left(\phi_1, K\right) \subseteq \Merg\left(\Phi, K\right)$, we have $\tau\mu_\xi$ is $\Phi$-ergodic for any $\xi$. Since $\nu$ is an extreme point, we have $\tau \mu_\xi = \nu$ for $\kappa$-a.e. $\xi$. Take a $\mu_\xi$ and item \eqref{item:4.1.4} is proved.

    Next, we prove $\phi_{[0,1]}\supp(\mu) = \supp(\tau\mu)$ for any $\mu \in \M(K)$. Fix $\mu \in \M(K)$. First, we prove $\phi_{[0,1]}\supp(\mu) \subseteq \supp(\tau\mu)$. Fix $y \in \phi_{[0,1]}\supp(\mu)$.  By definition, there exist $x \in \supp(\mu)$ and $t_0 \in [0,1]$ such that $y = \phi_{t_0}(x)$. For any open neighborhood $U$ of $y$, it suffices to show that $(\tau\mu)(U) > 0$. Note that $(\tau\mu)(U) = \int_0^1 \mu\left(\phi_{-t}(U)\right)  \dif t$. Since the flow is continuous, the map $(t, z) \mapsto \phi_t(z)$ is continuous on $\Rb \times K$. Since $\phi_{t_0}(x) \in U$ and $U$ is open, there exists an open interval $I$ of $t_0$ and an open neighborhood $V$ of $x$ such that $\phi_t(z) \in U$ for any $t \in I$ and any $z \in V$, which implies that $V \subseteq \phi_{-t}(U)$ for any $t \in I$. Since $x \in \supp(\mu)$ and $V$ is an open neighborhood of $x$, we have $\mu(V) > 0$. Consequently, for any $t \in I \cap [0,1]$, we have $\mu(\phi_{-t}(U)) \ge \mu(V) > 0$. Since the integrand $\mu(\phi_{-t}(U))$ is non-negative everywhere on $[0,1]$ and strictly bounded below by $\mu(V) > 0$ on the interval $I \cap [0,1]$ (which has positive Lebesgue measure), the integral is strictly positive. Thus, $(\tau\mu)(U) > 0$.

    Second, we prove $\supp(\tau\mu) \subseteq \phi_{[0,1]}\supp(\mu)$. Fix $y \in \supp(\tau\mu)$. Assume $y \notin \phi_{[0,1]}\supp(\mu)$. Since $\supp(\mu)$ is closed and $[0,1]$ is compact, the image $\phi_{[0,1]}\supp(\mu)$, is also a compact subset of $K$ by the continuity of $\Phi$. Since $\phi_{[0,1]}\supp(\mu)$ is closed and $y \notin \phi_{[0,1]}\supp(\mu)$, there exists an open neighborhood $U$ of $y$ such that $U \cap \phi_{[0,1]}\supp(\mu) = \emptyset$, which implies $\phi_{-t}(U) \cap \supp(\mu) = \emptyset$ for any $t \in [0,1]$. Then $\mu\left(\phi_{-t}(U)\right) = 0$ for any $t \in [0,1]$. Thus, $(\tau\mu)(U) = \int_0^1 \mu(\phi_{-t}(U)) \, dt = \int_0^1 0 \, dt = 0$, which is a contradiction to $y \in \supp(\tau\mu)$. Therefore, $\supp(\tau\mu) \subseteq \phi_{[0,1]}\supp(\mu)$. Then item (\ref{item:4.1.5}) is proved.

    Next, we prove item (\ref{item:4.1.6}). For any $\mu \in \M\left(\phi_1, K\right)$, since $\tau$ is affine, we have
    \[h_{\tau\mu}\left(\phi_1\right) = \int_0^1 h_{\mu \circ \phi_{-t}}\left(\phi_1\right)  \dif t.\]
    It suffices to prove $h_{\mu \circ \phi_{-t}}\left(\phi_1\right) = h_{\mu}\left(\phi_1\right)$ for any $t \in (0,1)$. For any $t \in (0,1)$ and any finite measurable partition $\xi$ of $K$, we have
    \[H_{\mu \circ \phi_{-t}}(\xi) = \sum_{A \in \xi}-\mu\left(\phi_{-t}A\right)\log \mu\left(\phi_{-t}A\right) = H_{\mu}\left(\phi_{-t}\xi\right)\]
    where $\phi_{-t} \xi :=\left\{\phi_{-t}A:A\in\xi\right\}$ is a finite measurable partition of $K$. Then
    \begin{align*}
        h_{\mu \circ \phi_{-t}}\left(\phi_1,\xi\right) & = \lim_{n \to \infty} \frac{1}{n} H_{\mu \circ \phi_{-t}}\left(\bigvee_{i=0}^{n-1}\phi_{-i}\xi\right)         \\
       & = \lim_{n \to \infty} \frac{1}{n} H_{\mu} \left(\phi_{-t}\bigvee_{i=0}^{n-1}\phi_{-i}\xi\right)               \\
       & = \lim_{n \to \infty} \frac{1}{n} H_{\mu} \left(\bigvee_{i=0}^{n-1}\phi_{-i}\left(\phi_{-t}\xi\right) \right) \\
       & = h_{\mu}\left(\phi_1, \phi_{-t}\xi\right).
    \end{align*}
    On the one hand,
    \begin{equation}\label{eq:4.2}
        h_{\mu \circ \phi_{-t}}\left(\phi_1\right) = \sup_{\xi}h_{\mu \circ \phi_{-t}}\left(\phi_1,\xi\right) = \sup_{\xi} h_{\mu}\left(\phi_1, \phi_{-t}\xi\right) \leq \sup_{\xi} h_{\mu}\left(\phi_1, \xi\right) = h_{\mu}\left(\phi_1\right).
    \end{equation}
    On the other hand, since $\mu$ is $\phi_1$-invariant, we have $\mu = \left(\mu \circ \phi_{-t} \right) \circ \phi_{-(1-t)}$. Similar to (\ref{eq:4.2}), we have
    \[h_{\mu}\left(\phi_1\right) = h_{\left(\mu \circ \phi_{-t} \right) \circ \phi_{-(1-t)}}\left(\phi_1\right) \leq h_{\mu \circ \phi_{-t}}\left(\phi_1\right),\]
    which implies $h_{\mu \circ \phi_{-t}}\left(\phi_1\right) = h_{\mu}\left(\phi_1\right)$. Then item (\ref{item:4.1.6}) is proved.

    Next, we prove item (\ref{item:4.1.7}). $\tau(V_{\phi_1}(x))\subseteq V_\Phi(x)$ follows from the same calculation used in the proof of item (\ref{item:4.1.3}). 
    To prove the reverse inclusion, let $\nu \in V_\Phi(x)$. Then there exists an increasing sequence of positive real numbers $\left\{t_k:k \in \Nb^+\right\}$ tending to infinity such that
    \[\lim_{k \to \infty} \cE_{t_k}(x,\Phi) = \nu.\]
    By (\ref{eq:phi_to_Phi2}), we obtain
    \[\lim_{k \to \infty} \tau\cE_{\left\lfloor t_k\right\rfloor}\left(x, \phi_{1} \right) = \nu.\]
    Take a limit point $\mu'$ of $\left\{\cE_{\left\lfloor t_k\right\rfloor}\left(x,\phi_{1}\right): k \in \Nb^+\right\}$ and hence $\mu ' \in V_{\phi_1}(x)$. Since $\tau$ is continuous, $\tau\mu' = \nu$, which implies that $V_{\Phi}(x) \subseteq \tau\left(V_{\phi_1}(x)\right)$. Then item (\ref{item:4.1.7}) is proved.

    Next, we prove item (\ref{item:4.1.8}). Fix a nonempty compact connected $\mathcal{K} \subseteq \M(\phi_1,K)$. For any $x \in G_{\mathcal{K}}\left(\phi_1\right)$, we have $V_{\phi_1}(x) = \mathcal{K}$. By item (\ref{item:4.1.7}), $V_\Phi(x) = \tau \mathcal{K}$ and hence $x \in G_{\tau \mathcal{K}}(\Phi)$. Since $G_{\tau \mathcal{K}}(\Phi)$ is $\Phi$-invariant, we have $\phi_{[0,1]}G_{\mathcal{K}}\left(\phi_1\right) \subseteq G_{\tau \mathcal{K}}(\Phi)$. Then item (\ref{item:4.1.8}) is proved.


\end{proof}
\begin{remark*}
    \begin{enumerate}
        \item In item (\ref{item:4.1.3}), $\tau$ is not necessarily injective on $\M\left(\phi_1, K\right)$. In item (\ref{item:4.1.4}), measures in $\tau^{-1}\Merg(\Phi,K)$ are not necessarily $\phi_1$-ergodic. A closed orbit $\gamma$ with period $1$ gives a counterexample. Indeed, on such an orbit, $\phi_1$ is the identity on $\gamma$.  Hence, every probability measure supported on $\gamma$ is $\phi_1$-invariant, while the averaging map $\tau$ sends all of them to $m_\gamma=\int_0^1\delta_{\phi_t(p)}\,dt$,
         the normalized time measure on $\gamma$, which is $\Phi$-ergodic.
        \item This lemma can be generalized to the case of $\tau_t(\mu) = \frac{1}{t}\int_0^t \mu \circ \phi_{-s} \dif s$ for any $t \in \Rb^+$ with some adjustments of the parameters.
        \item The inclusion in item (\ref{item:4.1.8}) can be strict. Indeed, consider the suspension flow over the full shift \((\Sigma,\sigma)\) with constant roof function \(1\). Let \(\mu\in\mathcal M_{\mathrm{erg}}(\sigma,\Sigma)\) and \(K=\{\mu\otimes\mathrm{Leb}\}\). Then
        $$ G_K(\phi_1)=\varnothing, \qquad G_K(\Phi)=G_\mu(\sigma)\times[0,1)\neq\varnothing. $$
        Thus,
        $$ \phi_{[0,1]}G_K(\phi_1)=\varnothing \subsetneq G_{\tau K}(\Phi)=G_K(\Phi). $$
        The details are given in Appendix~\ref{app:suspension-example}.
    \end{enumerate}
\end{remark*}

\subsubsection{Proof of the reduction lemma}
We now prove Lemma~\ref{lem:saturated-contained-qr}.  By Lemma~\ref{lem:tau}, the accumulation measures of the flow are obtained by applying $\tau$ to the accumulation measures of the time-one map, and the map $\tau$ preserves the entropy of $\phi_1$-invariant measures. Thus a point in $G_{\mathcal K}(\Phi)$ must have a discrete-time accumulation measure whose entropy is arbitrarily close to the infimum of $h_\mu(\phi_1)$ over $\mathcal K$.
\begin{proof}[Proof of Lemma~\ref{lem:saturated-contained-qr}]
Fix $x \in G_{\mathcal{K}}(\Phi)$ and $\ve > 0$. By definition, $V_\Phi(x)=\mathcal{K}$. By Lemma \ref{lem:tau}\eqref{item:4.1.7}, $\tau\left(V_{\phi_1}(x)\right) = \mathcal{K}$. By Lemma \ref{lem:tau}\eqref{item:4.1.6},
    \[\inf\left\{h_\mu\left(\phi_1\right):\mu \in V_{\phi_1}(x) \right\} = \inf\left\{h_\mu\left(\phi_1\right):\mu \in \mathcal{K} \right\}.\]
    It suffices to prove there exists $\mu \in V_{\phi_1}(x)$ such that $h_{\mu}\left(\phi_1\right) \leq \inf\left\{h_\mu\left(\phi_1\right):\mu \in V_{\phi_1}(x) \right\} + \ve$. Such a measure exists by the definition of the infimum.
\end{proof}

\section{Lower bounds of topological entropy for saturated sets}
The purpose of this section is to give the lower bound for saturated sets in a form that can be applied uniformly to the horseshoe approximations constructed later. The argument has two independent ingredients. The first is an entropy distribution principle, formulated through measure-theoretical lower entropy, which allows us to estimate the entropy of a compact set from a measure supported on it. The second is the uniform separation property, which supplies many orbit segments whose empirical measures stay close to prescribed ergodic measures. These two ingredients are then combined with specification on the approximating sets $\Lambda_n$ to construct points whose flow empirical measures have exactly the prescribed limit set.

In this section, suppose that $\Phi = \left(\phi_t\right)_{t \in \Rb}$ is a continuous flow on a compact metric space $K$, and $\Lambda \subseteq K$ is a compact $\Phi$-invariant subset.
\begin{theorem}\label{thm:5.1}
Assume the following.
\begin{enumerate}[label=\textup{(H\arabic*)}]
    \item The restriction $\Phi|_\Lambda$ is topologically transitive, asymptotically entropy-expansive, and satisfies
    \[
        \htop(\Phi|_\Lambda)<\infty.
    \]

    \item There exists a non-decreasing sequence of compact $\Phi$-invariant subsets $\{\Lambda_n\subseteq\Lambda:n\in\Nb^+\}$
    such that $\overline{\bigcup_{n\in\Nb^+}\Lambda_n}=\Lambda$, and $(\Lambda_n,\Phi)$ satisfies the specification property for every $n\in\Nb^+$.
\end{enumerate}

Let $\mathcal K$ be a nonempty compact connected subset of
$\M(\Phi,\Lambda)$. Assume that there exists a sequence
$\{\mu_i\}_{i\ge1}\subseteq\Merg(\Phi,\Lambda)$ satisfying

\begin{enumerate}[label=\textup{(M\arabic*)}]
    \item
        $\bigcap_{j\in\Nb^+}
        \overline{\{\mu_i:i\ge j\}}
        =\mathcal K$;
    \item
    $D(\mu_i,\mu_{i+1})\longrightarrow0
        \quad\text{as }i\longrightarrow\infty$;
    \item $\mu_i$ is supported on $\Lambda_i$ for every $i\in\Nb^+$.
\end{enumerate}

Then, for every open set $U\subseteq K$ with
$U\cap\Lambda\neq\varnothing$, one has
\[
    \htop\left(
        \Phi,
        G_{\mathcal K}(\Phi)\cap U\cap\Trans(\Phi,\Lambda)
    \right)
    \geq
    \liminf_{i\to\infty}h_{\mu_i}(\phi_1).
\]
\end{theorem}

\subsection{Technical Preparations}
We first present the two tools used in the proof of Theorem \ref{thm:5.1}. The first one converts lower estimates on Bowen balls for an auxiliary measure into a topological entropy lower bound. The second one provides the separated orbit segments from which the saturated set will be assembled.

\subsubsection{The measure-theoretical lower entropy}
Following ideas of Brin and Katok \cite{brin1983local}, Feng and Huang \cite{feng2012variational} introduced the \emph{measure-theoretical lower entropy} in discrete-time dynamical systems. Let $f:K \to K$ be a continuous map.
\begin{definition}
    Let $\mu \in \M(K)$. Define
    \[\lh_{\mu}(f) = \int \lh_{\mu}(f,x)\dif\mu\]
    where
    \[\lh_\mu(f,x) = \lim_{\ve \to 0} \liminf_{n \to \infty} - \frac{1}{n}\log \mu\left(B_n(x,\ve,f)\right).\]
\end{definition}
Feng and Huang \cite{feng2012variational}*{Theorem 1.2} proved the following variational principle of measure-theoretical lower entropy. If $Z \subseteq K$ is nonempty and compact, then
\[\htop\left(f,Z\right) = \sup\left\{\lh_\mu(f):\mu\in\M(K),\mu(Z)=1\right\}.\]
By Lemma \ref{lem:htop}, we have the following lemma.
\begin{lemma}\label{lem:lower_entropy}
    Let $\Phi = \left(\phi_t\right)_{t \in \Rb}$ be a continuous flow on a compact metric space $K$. Then for any nonempty and compact set $Z \subseteq K$, we have
    \[\htop\left(\Phi,Z\right) = \sup\left\{\lh_\mu\left(\phi_1\right) :\mu\in\M(K),\mu(Z)=1\right\}.\]
\end{lemma}

\subsubsection{The uniform separation property}
The preceding lemma will be applied to a measure carried by the set constructed through specification. To ensure that this measure has large lower entropy, we need many separated choices at each stage of the construction. This is provided by the uniform separation for the time-one map.

Let $f:K \to K$ be a continuous map, $n \in \Nb^+$ and $\ve > 0$. Two points $x,y \in K$ are called \emph{\((n,\ve)\)-separated} for $f$ if there is $0 \leq i \leq n-1$ such that $d\left(f^i x,f^i y\right) > \ve$. Let $0 < \delta \leq 1$. Two points $x,y \in K$ are called \emph{\((\delta,n,\ve)\)-separated} for $f$ if
\[  \left\lvert \left\{0 \leq i \leq n-1:d\left(f^i x,f^i y\right) > \ve \right\} \right\rvert \geq \delta n \]
where, for a finite set \(A\), \(|A|\) denotes its cardinality. A subset $E \subseteq K$ is called $(n,\ve)$-separated (respectively, $(\delta, n,\ve)$-separated) for $f$ if any two distinct points in $E$ are $(n,\ve)$-separated (respectively, $(\delta, n,\ve)$-separated) for $f$.

For any $F \subseteq \M(K)$, let
\[
    \cE_n^{-1}(F) := \left\{x \in K: \cE_n(x, f) \in F\right\}.
\]
Note that in this article, the empirical measures associated with $\cE_n^{-1}(F)$ are taken with respect to the continuous map $f$ or the time-1 map $\phi_1$ of the flow, and do not involve the empirical measures along the flow.
\begin{definition}
    We say $(K, f)$ has the \emph{uniform separation property} if for any $\eta > 0$, there exist $\delta>0$ and $\ve > 0$ such that for any $\mu \in \Merg(f,K)$ and any neighborhood $F$ of $\mu$ in $\M(K)$, there exists $N \in \Rb^+$ such that for any $n \geq N$, there exists a $\left(\delta, n, \ve\right)$-separated set $\Gamma_n \subset \cE_n^{-1}(F)$ for $f$ such that $\left\lvert \Gamma_n\right\rvert \geq e^{n\left(h_{\mu}\left(f\right) - \eta\right)}$. 
\end{definition}

Since $\delta>0$, any $\left(\delta, n, \ve\right)$-separated set for $f$ is also $\left(n, \ve\right)$-separated. By \cite{pfister2007topological}*{Theorem 3.1}, the uniform separation property is a consequence of asymptotic entropy-expansiveness. Furthermore, by \cite{hou2021entropy}*{Proposition 2.9}, the separated set $\Gamma_n$ can be chosen as a subset of $\supp(\mu)$ if the system is asymptotically entropy-expansive, which leads to the following lemma.

\begin{lemma}\label{lem:AEE-to-USP}
    If $(\Lambda,\Phi)$ is asymptotically entropy-expansive, then $\left(\Lambda,\phi_1\right) $ has the uniform separation property. Moreover, for any $\eta > 0$, there exists $\ve > 0$ such that for any $\mu \in \Merg(\phi_1,\Lambda)$ and any neighborhood $F$ of $\mu$ in $\M(\Lambda)$, there exists $N \in \Rb^+$ such that for any $n \geq N$, there exists a $\left(n, \ve\right)$-separated set $\Gamma_n \subset \cE_n^{-1}(F) \cap \supp(\mu)$ for $\phi_1$ such that $ \left\lvert \Gamma_n\right\rvert \geq e^{n\left(h_{\mu}\left(\phi_1\right) - \eta\right)}$.
\end{lemma}

\subsection{Proof of Theorem \ref{thm:5.1}}
We now prove Theorem \ref{thm:5.1}. At the $i$-th stage, we construct a long block of orbit segments typical for $\mu_i$, whose lengths are chosen so that the transition times and the previous stages are negligible. The additional orbit segments of transitive points in the sets $\Lambda_i$ force the final points to be transitive in $\Lambda$, while the sequence of measures $\{\mu_i\}$ determines the limit set of empirical measures.

\begin{proof}[Proof of Theorem \ref{thm:5.1}]
    \textbf{Preliminaries.} Fix a nonempty open set $U \subseteq K$ with $U \cap \Lambda \ne \emptyset$. Since $\overline{\bigcup_{n \in \Nb^+} \Lambda_n} = \Lambda$ and $\{ \Lambda_{n} \subseteq \Lambda:n \in \mathbb{N}^{+} \}$ is non-decreasing, we have $U \cap \Lambda_n \ne \emptyset$ for $n \in \Nb^+$ large enough. Without loss of generality, assume $U \cap \Lambda_1 \ne \emptyset$. Fix $x \in U \cap \Lambda_1$ and $\ve_0 > 0$ such that $\overline{B\left(x,\ve_0\right)} \subseteq U$. We prove the case of $\liminf_{i \to \infty} h_{\mu_i}(\phi_1) > 0$ and give a brief explanation of the other cases at the end of the proof. It suffices to prove 
    \[\htop\left(\Phi, G_{\mathcal{K}}\left(\Phi\right) \cap \Trans(\Phi,\Lambda) \cap U\right) \geq \liminf_{i \to \infty} h_{\mu_i}(\phi_1) - 3\eta\]
    for any $\eta > 0$. Fix $0 < \eta < \liminf_{i \to \infty} h_{\mu_i}(\phi_1) / 3$ and $H = \liminf_{i \to \infty} h_{\mu_i}(\phi_1) - \eta$. Since $\bigcap_{j\in\Nb^+} \overline{\left\{\mu_i\right\}_{i=j}^\infty} = \mathcal{K}$, without loss of generality, assume $h_{\mu_i}(\phi_1) \geq H$ for any $i \in \Nb^+$ (otherwise, we take a subsequence by discarding the first finitely many terms). Let $\delta > 0$ be such that
    \begin{equation}\label{eq:eta}
        3\delta \leq \frac{\eta}{H-2\eta}.
    \end{equation}
    Since $(\Lambda,\Phi)$ is asymptotically entropy-expansive, $\left(\Lambda,\phi_1\right)$ satisfies the uniform separation property by Lemma \ref{lem:AEE-to-USP}. For this $\eta$, there exists $\ve^* > 0$ such that for any $\mu \in \Merg(\phi_1,\Lambda)$ and any neighborhood $F$ of $\mu$ in $\M(\Lambda)$, there exists $N \in \Nb^+$ such that for any $n \geq N$, there exists a $\left(n, \ve^*\right)$-separated set $\Gamma_n \subset \cE_n^{-1}(F) \cap \supp(\mu)$ for $\phi_1$ such that $\left\lvert \Gamma_n\right\rvert \geq e^{n\left(h_{\mu}\left(\phi_1\right)-\eta\right) }$.

    \medskip

    \textbf{Choice of parameters.} Take a sequence of real numbers $\left\{\ve_i\right\}_{i=1}^\infty$ satisfying
    \[0 < \ve_1 < \min\left\{\ve_0/2,\ve^*/6\right\}\]
    and
    \[\ve_{i+1} = \ve_i / 2\]
    for any $i \in \Nb^+$. Since $\overline{\bigcup_{n \in \Nb^+} \Lambda_n} = \Lambda$, assume $\Lambda \subseteq B(\Lambda_n, \ve_n)$ without loss of generality.

    Since $(\Lambda_n,\Phi)$ satisfies the specification property, there is a transitive point $x_n \in \Lambda_n$ of $(\Lambda_n,\Phi)$ and $R_n \in \Rb^+$ such that $\Lambda_n \subseteq B(\{\phi_{t}x_n:0 \leq t \leq R_n\}, \ve_n)$. Thus, we can take a sequence of positive integers $\left\{R_i\right\}_{i=1}^\infty$.

    Let $\tau:\M(\Lambda) \to \M(\Lambda)$ be the averaging map defined by
    \[\tau(\mu) = \int_{0}^{1}\mu \circ \phi_{-t} \dif t\]
    as in Lemma \ref{lem:tau}, which is continuous. Let $\delta_0 = \delta$. Since $\M(\Lambda)$ is compact, $\tau$ is uniformly continuous and hence for any $i \in \Nb^+$, there exists $0 < \delta_i < \min\left\{\ve_i, \delta_{i-1} / 2, \delta\right\}$ such that
    \begin{equation}\label{eq:tau}
        D\left(\tau\mu,\tau\nu\right) \leq \ve_i \text{\ for\ any\ } \mu,\nu \in \M(\Lambda) \text{\ with\ } D\left(\mu,\nu\right) \leq 7\delta_i.
    \end{equation}

    For any $i \in \Nb^+$, there exists $\nu_i \in \Merg\left(\phi_1,\Lambda\right)$ such that $\tau\nu_i = \mu_i$ and $\supp\left(\nu_i\right) \subseteq \supp\left(\mu_i\right) \subseteq \Lambda_i$ by Lemma \ref{lem:tau}.

    For any $i \in \Nb^+$, let $T_i$ be the transition time determined by the specification property of $\left(\Lambda_{i},\Phi\right)$ with scale $\delta_i$. By Definition \ref{def:specification}, we can assume $T_i \in \Nb^+$ for any $i \in \Nb^+$.

    Let $t_0 = 1$. For any $i \in \Nb^+$, by the uniform separation property of $\left(\Lambda,\phi_1\right) $ and Lemma \ref{lem:AEE-to-USP}, there exists a positive integer
    \[t_i > \max\left\{T_i/\delta_i,T_i/\delta, t_{i-1}, R_i, N^{\mathrm{USP}}_i\right\}\]
    and a $\left(t_i,\ve^*\right)$-separated set
    \begin{equation}\label{eq:Gamma_D}
        \Gamma_{t_i} \subseteq \cE_{t_i}^{-1}(B\left(\nu_i,\delta_i\right)) \cap \supp\left(\nu_i\right) \subseteq \cE_{t_i}^{-1}(B\left(\nu_i,\delta_i\right)) \cap \Lambda_i
    \end{equation}
    for $\phi_1$ such that
    \begin{equation}\label{eq:Gamma_cardianlity}
        \left\lvert \Gamma_{t_i}\right\rvert \geq e^{t_i\left(h_{\nu_i}\left(\phi_1\right) - \eta\right)}
    \end{equation}
    where $N^{\mathrm{USP}}_i$ is the positive integer determined by the uniform separation property with respect to $\nu_i$ and $B\left(\nu_i,\delta_i\right)$.

    Next, we define $\{N_i\}_{i\in\Nb^+}$ and $\{L_i\}_{i\in\Nb}$ recursively. Set $L_0:=0$. Once $L_{i-1}$ has been defined, choose $N_i\in\Nb^+$ satisfying that
    \begin{equation}\label{eq:N_D}
        N_i \geq \frac{L_{i-1}+T_i+2T_{i+1}+t_i+2t_{i+1}}{\delta_i t_i - T_i}.
    \end{equation}
    This choice is possible because $\delta_i t_i>T_i$. Let
    \[
        L_i =
        \begin{cases}
            \left(N_1+2\right)\left(t_1+T_1\right), & i=1,\\
            L_{i-1}+\left(N_i+1\right)\left(T_i+t_i\right), & i\geq 2.
        \end{cases}
    \]

    Let $\Gamma^{N_i}_{t_i} = \prod_{j=1}^{N_i} \Gamma_{t_i}$ for any $i \in \Nb^+$.

    \medskip

    \textbf{Construction of $G$.} We construct a compact subset $G$ of $G_{\mathcal{K}}(\Phi)$ inductively.

    \textit{Step(1):} Recall that we have fixed a point $x \in U \cap \Lambda_{1}$ and $x_{1}$ is a transitive point of $(\Lambda_{1},\Phi)$.  For any $\overline{y^1} = \left(y_1^1,\dots,y_{N_1}^1\right) \in \Gamma^{N_1}_{t_1}$ and any $\ve > 0$, let
    \[G_1\left(\overline{y^1},\ve\right) = B_{t_1}(x,\ve) \cap \phi_{-T_1-t_1} B_{t_1}(x_1,\ve) \cap \bigcap_{j=0}^{N_1-1} \phi_{-(j+2)\left(t_1+T_1\right)}B_{t_1}\left(y_{j+1}^1,\ve\right),\]
    denote the points that $\ve$-shadow the orbit segments
    \[\left(x,t_1\right)
        \overset{T_1}{\longrightarrow}\left(x_1,t_1\right)
        \overset{T_1}{\longrightarrow}\left(y^1_1,t_1\right)
        \overset{T_1}{\longrightarrow}\left(y^1_2,t_1\right)
        \overset{T_1}{\longrightarrow}\dots
        \overset{T_1}{\longrightarrow}\left(y^1_{N_1},t_1\right).\]
    By the specification property of $\left(\Lambda_1, \Phi\right)$ and choice of $T_{1}$ and $\delta_{1}$, the set $G_1\left(\overline{y^1},\delta_1\right) \cap \Lambda_1$ is nonempty. Let
    \[G_1 = \bigcup_{\overline{y^1} \in \Gamma_{t_1}^{N_1}} G_1\left(\overline{y^1}, 2\delta_1\right) \cap \Lambda.\]
    and hence $G_1$ is nonempty. For any $\overline{y^1} = \left(y_1^1,\dots,y_{N_1}^1\right) \in \Gamma^{N_1}_{t_1} \subseteq \Lambda_1$, fix $z_1\left(\overline{y^1}\right) \in G_1\left(\overline{y^1},\delta_1\right) \cap \Lambda_1$. 
    Let $F_1 = \left\{z_1\left(\overline{y^1}\right) : \overline{y^1} \in \Gamma_{t_1}^{N_1}\right\} \subseteq \Lambda_1$. Distinct choices of \(\overline y^{\,1}\) yield distinct points \(z_1(\overline y^{\,1})\), by the \((t_1,\varepsilon_*)\)-separation of \(\Gamma_{t_1}\) and the \(\delta_1\)-shadowing property, and hence $\left\lvert F_1\right\rvert = \left\lvert\Gamma_{t_1} \right\rvert^{N_1} $. Note $F_1 \subseteq G_1$. Since $\Gamma_{t_1}$ is $\left(t_1, \ve^*\right)$-separated for $\phi_1$, we have $F_1$ is $\left(L_1,\ve^*-2\delta_1\right)$-separated for $\phi_1$.

    \medskip

    \textit{Step($k$) for any $k \ge 2$:} For any $z_{k-1} = z_{k-1}\left(\overline{y^1},\dots, \overline{y^{k-1}}\right) \in F_{k-1} \subseteq \Lambda_{k-1}$, any $\overline{y^k} = \left(y^k_1,\dots,y^k_{N_k}\right) \in \Gamma_{t_k}^{N_k} \subseteq \Lambda_k$ and any $\ve > 0$, let
    \[G_k\left(z_{k-1},\overline{y^k},\ve\right) = B_{L_{k-1}}\left(z_{k-1},\ve\right) \cap   \phi_{-L_{k-1}-T_k}B_{t_k}\left(x_k,\ve\right) \cap \left(\bigcap_{j=0}^{N_k-1} \phi_{-L_{k-1}-T_{k}-(j+1)\left(t_k+T_k\right)}B_{t_k}\left(y_{j+1}^k,\ve\right)\right)\]
    denote the points that $\ve$-shadow the orbit segments
    \[\left(z_{k-1},L_{k-1}\right)
        \overset{T_k}{\longrightarrow}\left(x_k,t_k\right)
        \overset{T_k}{\longrightarrow}\left(y^k_1,t_k\right)
        \overset{T_k}{\longrightarrow}\left(y^k_2,t_k\right)
        \overset{T_k}{\longrightarrow}\dots
        \overset{T_k}{\longrightarrow}\left(y^k_{N_k},t_k\right).\]
    By the specification property of $\left(\Lambda_k, \Phi\right)$, the set $G_k\left(z_{k-1},\overline{y^k},\delta_k\right) \cap \Lambda_k$ is nonempty. Let
    \[G_k = \bigcup_{z_{k-1} \in F_{k-1}, \overline{y^k} \in \Gamma_{t_k}^{N_k}} G_k\left(z_{k-1},\overline{y^k}, 2\delta_k\right) \cap \Lambda\]
    and hence $G_k$ is nonempty. Note $G_k \subseteq G_{k-1}$. For any $\left( \overline{y^1},\dots,\overline{y^k}\right)  \in \Gamma^{N_1}_{t_1} \times \cdots \times \Gamma^{N_k}_{t_k}$, fix
    \[z_k\left(\overline{y^1},\dots,\overline{y^k}\right) \in G_k\left(z_{k-1}\left(\overline{y^1},\dots, \overline{y^{k-1}} \right) ,\overline{y^k},\delta_k\right) \cap \Lambda_k.\]
    Let
    \[F_k = \left\{z_k\left(\overline{y^1},\dots,\overline{y^k}\right) : \overline{y^i} \in \Gamma_{t_i}^{N_i}, \forall 1 \leq i \leq k\right\} \subseteq \Lambda_k \]
    and hence $\left\lvert F_k\right\rvert = \left\lvert F_{k-1}\right\rvert \left\lvert \Gamma_{t_k}^{N_k}\right\rvert = \prod_{i=1}^{k}\left\lvert \Gamma_{t_i}\right\rvert^{N_i} $. Note $F_k \subseteq G_k$. Since $F_{k-1}$ is $\left(L_{k-1},\ve^*-2\sum_{i=1}^{k-1}\delta_i\right)$-separated for $\phi_1$ and $\Gamma_{t_k}$ is $\left(t_k,\ve^*\right)$-separated for $\phi_1$, we have $F_k$ is $\left(L_k,\ve^*-2\sum_{i=1}^{k}\delta_i\right)$-separated for $\phi_1$.

    Then we have a sequence of nonempty subsets $\left\{G_i:i\in\Nb^+\right\}$. Let $G=\bigcap_{i=1}^\infty \overline{G_i} \subseteq \Lambda$. By the Cantor intersection theorem, $G$ is nonempty and compact. 

    Since \(G_i\) is a finite union of cylinders, every \(z\in G\) belongs to the closure of a cylinder at each level. The separation of the sets \(\Gamma_{t_i}\), together with \(4\delta_i<\varepsilon_*\), makes this cylinder unique. By nesting, these choices determine a compatible infinite itinerary \((\overline {y^{\,i}}(z))_{i\ge1}\) with \(\overline{y^{i}}(z)=\left(y^i_1,\dots,y^i_{N_i}\right) \in\Gamma_{t_i}^{N_i}\), such that the orbit of \(z\) shadows the prescribed segments successively, with accuracy \(2\delta_i\) at stage \(i\). That is,
    $z$  $2\delta_1$-shadows the orbit segments
    \[\left(x,t_1\right)
        \overset{T_1}{\longrightarrow}\left(x_1,t_1\right)
        \overset{T_1}{\longrightarrow}\left(y^1_1,t_1\right)
        \overset{T_1}{\longrightarrow}\left(y^1_2,t_1\right)
        \overset{T_1}{\longrightarrow}\dots
        \overset{T_1}{\longrightarrow}\left(y^1_{N_1},t_1\right),\]
    then $2\delta_2$-shadows the orbit segments
    \[
        \overset{T_2}{\longrightarrow}\left(x_2,t_2\right)
        \overset{T_2}{\longrightarrow}\left(y^2_1,t_2\right)
        \overset{T_2}{\longrightarrow}\left(y^2_2,t_2\right)
        \overset{T_2}{\longrightarrow}\dots
        \overset{T_2}{\longrightarrow}\left(y^2_{N_2},t_2\right),\]
    and so on for each $i \in \Nb^+$.

    \medskip

    \textbf{Properties of $G$.} We prove that $G \subseteq G_{\mathcal{K}}(\Phi) \cap \Trans(\Phi,\Lambda) \cap U$.

    \textit{Inclusion within the Transitive Domain.} First, we prove that $G \subseteq U \cap \Trans(\Phi,\Lambda)$. By $G \subseteq \overline{G_1}$ and the definition of $G_1$, we have $G \subseteq \overline{B(x,2\delta_1)} \subset \overline{B(x,2\ve_1)} \subset \overline{B(x,\ve_0)} \subset U$. For every $z \in G$, the orbit segment shadowing $(x_{n},t_{n})$ is $(2\delta_{n}+2\varepsilon_{n})$-dense in $\Lambda$. These segments occur at times tending to infinity, and their density scales tend to zero. Hence $\omega_{\Phi}(z)=\Lambda$, so $G \subseteq \Trans(\Phi,\Lambda)$.

    \medskip

    \textit{Inclusion within the Saturated Set.} 
     Next, we prove $G \subseteq G_{\mathcal{K}}(\Phi)$. It suffices to prove $V_{\Phi}(z)=\mathcal K$ for every $z\in G$. 
     Fix $z\in G$, and let $(\overline{y^i} )_{i \geq 1}$ be its compatible itinerary constructed above. For each $i\ge1$, write $z_i=z_i(\overline{y^1},\ldots,\overline{y^i})\in F_i$.
    Let $s\in\mathbb N^+$, and choose $k$ such that $L_k\le s<L_{k+1}$.

    If $L_{k} \leq s < L_{k} + 2T_{k+1} + 2t_{k+1}$, by Remark \ref{rmk:D}, \eqref{eq:Gamma_D} and \eqref{eq:N_D},
    \begin{align*}
        D\left(\cE_s(z,\phi_1),\nu_k\right) & \leq D\left(\cE_{s}(z,\phi_1), \frac{1}{N_k}\sum_{j=1}^{N_k}\cE_{t_k}\left(y^k_j,\phi_1\right)\right)+ D\left(\frac{1}{N_k}\sum_{j=1}^{N_k}\cE_{t_k}\left(y^k_j,\phi_1\right), \nu_k\right) \\
                                     & \leq 2\delta_k + 2\delta_k + \delta_k = 5\delta_k.
    \end{align*}
    By \eqref{eq:tau}, we have $D\left(\cE_s(z,\Phi), \mu_k\right)=D\left(\tau\cE_s(z,\phi_{1}), \tau\nu_k\right) \leq \ve_k$.

    If $L_{k} + 2T_{k+1} + 2t_{k+1} \leq s < L_{k+1}$, let
    \[m = \left\lfloor \frac{s-L_{k} - T_{k+1} - t_{k+1}}{t_{k+1}+T_{k+1}}\right\rfloor \geq 1\]
    where $\left\lfloor \cdot \right\rfloor $ denotes the maximal integer not larger than the given number. By Remark \ref{rmk:D}, \eqref{eq:Gamma_D} and \eqref{eq:N_D},
    \begin{align*}
         & \phantom{\leq} D\left(\cE_s(z,\phi_1),\frac{N_kt_k}{N_kt_k+mt_{k+1}}\nu_k + \frac{mt_{k+1}}{N_kt_k+mt_{k+1}}\nu_{k+1}\right)                                                                                                                             \\
         & \leq D\left(\cE_{s}(z,\phi_1), \frac{N_kt_k}{N_kt_k+mt_{k+1}}\left(\frac{1}{N_k}\sum_{j=1}^{N_k}\cE_{t_k}\left(y^k_j,\phi_1\right)\right)  + \frac{mt_{k+1}}{N_kt_k+mt_{k+1}}\left(\frac{1}{m}\sum_{j=1}^{m}\cE_{t_{k+1}}\left(y^{k+1}_j,\phi_1\right)\right)  \right) \\
         & + \frac{N_kt_k}{N_kt_k+mt_{k+1}}D\left(\frac{1}{N_k}\sum_{j=1}^{N_k}\cE_{t_k}\left(y^k_j,\phi_1\right), \nu_k\right)                                                                                                                                     \\
         & + \frac{mt_{k+1}}{N_kt_k+mt_{k+1}}D\left(\frac{1}{m}\sum_{j=1}^{m}\cE_{t_{k+1}}\left(y^{k+1}_j,\phi_1\right), \nu_{k+1}\right)                                                                                                                           \\
         & \leq 2\delta_k + 2\delta_k + \delta_k + \delta_{k+1} \leq 6\delta_k.
    \end{align*}
    By \eqref{eq:tau} and the affinity of $\tau$ from Lemma \ref{lem:tau}, we have
    \begin{align*}
         & \phantom{=} D\left(\cE_s(z,\Phi), \frac{N_kt_k}{N_kt_k+mt_{k+1}}\mu_k + \frac{mt_{k+1}}{N_kt_k+mt_{k+1}}\mu_{k+1}\right)             \\
         & = D\left(\tau\cE_s(z,\phi_{1}), \tau\left(\frac{N_kt_k}{N_kt_k+mt_{k+1}}\nu_k + \frac{mt_{k+1}}{N_kt_k+mt_{k+1}}\nu_{k+1}\right) \right) \\
         & \leq \ve_k.
    \end{align*}

    Combining both cases, we obtain $D\bigl(\mathcal{E}_s(z,\Phi),\mu_k\bigr)
    \leq \varepsilon_k+D(\mu_k,\mu_{k+1})$
    for every integer $s\in[L_k,L_{k+1})$.
    For a real number $s\in[L_k,L_{k+1})$, we have $\lfloor s\rfloor\in[L_k,L_{k+1})$, since the numbers $L_k$
    are integers. By Lemma~\ref{lem:D},
    \[
    D\bigl(\mathcal{E}_s(z,\Phi),
       \mathcal{E}_{\lfloor s\rfloor}(z,\Phi)\bigr)<\frac{1}{s}.
    \]
    Consequently,
    \[
    D\bigl(\mathcal{E}_s(z,\Phi),\mu_k\bigr)
    \leq \varepsilon_k+D(\mu_k,\mu_{k+1})+\frac{2}{s}.
    \]
    Since $\lim_{i \to \infty}\ve_i = 0$ and $\lim_{i \to \infty}D(\mu_i, \mu_{i+1}) = 0$, the set $\left\{\cE_s(z, \Phi):s \in \Rb^+\right\} $ has the same limit points as $\left\{\mu_i:i \in \Nb^+\right\}$. Since $\bigcap_{j\in\Nb^+} \overline{\left\{\mu_i\right\}_{i=j}^\infty} = \mathcal{K}$, we have $V_{\Phi}(z) = \mathcal{K}$. Therefore, $G \subseteq G_{\mathcal{K}}(\Phi)$.

    \medskip

    \textit{Lower bound of topological entropy.} Finally, we prove that
    \[\htop\left(\Phi,G\right) \geq H - 2\eta.\]
    Let
    \[\omega_k := \frac{1}{\left\lvert F_k\right\rvert}\sum_{z_k \in F_k}\delta_{z_k}\]
    supported on $\Lambda_k$. Take a converging subsequence $\{ \omega_{k_{n}} \}$ of $\{ \omega_{k} \}$ and let $\omega=\lim_{n \rightarrow \infty}\omega_{k_{n}}$. Hence, $\omega$ is supported on $\Lambda$. Given $p \in \Nb^+$, we have $F_{p+q} \subseteq G_{p + q}$ and hence $\omega_{p+q}\left(G_{p+q}\right) = 1$ for any $q \in \Nb^+$, which implies that $\omega_{p + q}\left(G_p\right) = 1$ for any $q \in \Nb^+$. Therefore, $\omega\left(\overline{G_p}\right) \geq \limsup_{n \to \infty}\omega_{k_n}\left(\overline{G_p}\right) = 1$ and hence $\omega\left(G\right) = \lim_{i \to \infty} \omega\left(\overline{G_{i}}\right) = 1$.

    Fix $z \in G$ and a positive integer $s \geq L_1$. There exists $k \in \Nb^+$ and $1 \leq j \leq N_{k+1}+1$ such that
    \[L_k + (j-1)\left(t_{k+1}+T_{k+1}\right) \leq s < L_{k} +j\left(t_{k+1}+T_{k+1}\right).\]
    Since $z \in G$, there exists $\left(\overline{y^1}',\dots,\overline{y^{k+1}}'\right) \in \Gamma_{t_1}^{N_1} \times \dots \times \Gamma_{t_{k+1}}^{N_{k+1}}$ such that
    \[z \in \overline{G_{k+1}\left(z_{k}\left(\overline{y^1}',\dots,\overline{y^{k}}'\right) ,\overline{y^{k+1}}',2\delta_{k+1}\right)}.\]
    For any $k_n \geq k + 1$ and any $\left(\overline{y^1},\dots,\overline{y^{k_n}}\right) \in \Gamma_{t_1}^{N_1} \times \dots \times \Gamma_{t_{k_n}}^{N_{k_n}}$, if
    \[\left(\overline{y^1},\dots,\overline{y^{k}},\left(y_1^{k+1},\dots,y_{j-2}^{k+1}\right) \right) \ne \left(\overline{y^1}',\dots,\overline{y^{k}}',\left({y_1^{k+1}}',\dots,{y_{j-2}^{k+1}}'\right)\right)\]
    where we denote $\left(y_1^{k+1},\dots,y_{j-2}^{k+1}\right) = \left({y_1^{k+1}}',\dots,{y_{j-2}^{k+1}}'\right) = \emptyset$ if $1 \leq j \leq 2$, then $z$ and $z_{k_n}\left(\overline{y^1},\dots,\overline{y^{k_n}}\right) $ are $\left(L_{k}+\left(j-1\right) \left(t_{k+1}+T_{k+1}\right),\ve^*-2\delta_1-\sum_{i=1}^{k_n}\delta_i\right)$-separated for $\phi_1$, hence $\left(L_{k}+\left(j-1\right) \left(t_{k+1}+T_{k+1}\right),\ve^*/3\right)$-separated for $\phi_1$. It implies that
    \[z_{k_n}\left(\overline{y^1},\dots,\overline{y^{k_n}}\right) \notin B_{L_{k}+\left(j-1\right) \left(t_{k+1}+T_{k+1}\right)}\left(z,\ve^*/3, \phi_1\right).\]
    By \eqref{eq:Gamma_cardianlity},
    \begin{align*}
        \omega_{k_n}\left(B_s\left(z,\ve^*/3, \phi_1\right) \right) & \leq \omega_{k_n}\left(B_{L_k + (j-1)\left(t_{k+1}+T_{k+1}\right)}\left(z,\ve^*/3, \phi_1\right)  \right)                                                                                                      \\
        & \leq \frac{\left\lvert \Gamma_{t_{k+1}}\right\rvert^{N_{k+1}-(j-2)_{+}} \cdot \prod_{i=k+2}^{k_n}\left\lvert \Gamma_{t_i}\right\rvert^{N_i}}{\prod_{i=1}^{k_n}\left\lvert \Gamma_{t_i}\right\rvert^{N_i} } \\
        & = \frac{1}{\left\lvert \Gamma_{t_{k+1}}\right\rvert^{(j-2)_{+}}\cdot\prod_{i=1}^{k}\left\lvert \Gamma_{t_i}\right\rvert^{N_i}}                                                                         \\
        & \leq e^{-t_{k+1}(j-2)_{+}\left(h_{\nu_{k+1}}\left(\phi_1\right) - \eta\right)-\sum_{i=1}^{k}t_iN_i\left(h_{\nu_i}\left(\phi_1\right) - \eta\right)}                                                    \\
        & \leq e^{-\left(t_{k+1}(j-2)_{+}+\sum_{i=1}^{k}t_iN_i\right)\left(H-\eta\right)}
    \end{align*}
    where $(a)_{+} := \max\left\{a,0\right\} $. By \eqref{eq:eta},\eqref{eq:N_D}
    \begin{align*}
        \frac{s-t_{k+1}(j-2)_{+}-\sum_{i=1}^{k}t_iN_i}{t_{k+1}(j-2)_{+}+\sum_{i=1}^{k}t_iN_i}
         & \leq \frac{t_1+T_1 + \sum_{i=1}^{k}\left(T_i\left(N_i+1\right) + t_i\right) + 2t_{k+1}+jT_{k+1}}{t_{k+1}(j-2)_{+}+\sum_{i=1}^{k}t_iN_i}                                             \\
         & \leq \frac{t_1+T_1}{t_{k+1}(j-2)_{+}+\sum_{i=1}^{k}t_iN_i} + \frac{\sum_{i=1}^{k}\left(T_i\left(N_i+1\right) + t_i\right)}{t_{k+1}(j-2)_{+}+\sum_{i=1}^{k}t_iN_i}                   \\
         & + \frac{2t_{k+1}+jT_{k+1}}{t_{k+1}(j-2)_{+}+\sum_{i=1}^{k}t_iN_i}                                                                                                                   \\
         & \leq \frac{t_1+T_1}{t_1N_1} + \frac{\sum_{i=1}^{k}\left(T_i\left(N_i+1\right) + t_i\right)}{\sum_{i=1}^{k}t_iN_i} + \frac{2t_{k+1}+jT_{k+1}}{t_{k+1}(j-2)_{+}+\sum_{i=1}^{k}t_iN_i} \\
         & \leq \delta + \delta + \delta = 3\delta \leq \frac{\eta}{H-2\eta}.
    \end{align*}
    Therefore,
    \[s(H-2\eta) \leq \left(t_{k+1}(j-2)_{+}+\sum_{i=1}^{k}t_iN_i\right)\left(H-\eta\right).\]
    Then
    \[\omega\left(B_s\left(z,r, \phi_1\right) \right) \leq \omega\left(B_s\left(z,\ve^*/3, \phi_1\right) \right) \leq \liminf_{n \to \infty}\omega_{k_n}\left(B_s\left(z,\ve^*/3, \phi_1\right) \right) \leq e^{-s(H-2\eta)}\]
    for any $0 < r < \ve^* / 3$, any $s \geq L_1$ and any $z \in G$. Therefore,
    \begin{align*}
        \lh_{\omega}\left(\phi_1,z\right) & = \lim_{r \to 0} \liminf_{n \to \infty} - \frac{1}{n}\log\omega\left(B_n\left(z,r, \phi_1\right) \right) \\
                                          & \geq H - 2\eta
    \end{align*}
    Then
    \[\lh_{\omega}\left(\phi_1\right) = \int\lh_{\omega}\left(\phi_1,z\right) \dif \omega(z) \geq H - 2\eta\]
    and hence $\htop\left(\Phi,G\right) \geq H - 2\eta$ by Lemma \ref{lem:lower_entropy}, which completes the proof by Lemma \ref{lem:htop} and the fact that $G \subseteq G_{\mathcal{K}}(\Phi) \cap \Trans(\Phi,\Lambda) \cap U$.

    For the case of $\liminf_{i \to \infty} h_{\mu_i}(\phi_1) = 0$, it suffices to prove that $G_{\mathcal{K}}\left(\Phi\right) \cap \Trans(\Phi,\Lambda) \cap U$ is nonempty. We can dispense with the restrictions on $\delta$ and $\eta$ as well as the quantitative entropy estimates at the end of this proof, and the construction establishing the non-emptiness of the saturated set still holds.
\end{proof}
\begin{remark}\label{rem:nontransitive-version}
    The assumption
    \[
        \overline{\bigcup_{n\in\Nb^+}\Lambda_n}=\Lambda
    \]
    is only used to ensure that the constructed points are transitive in \(\Lambda\) and to localize the construction in a prescribed nonempty relatively open subset of \(\Lambda\). Without this assumption, and without requiring transitivity or localization, all the remaining hypotheses of Theorem \ref{thm:5.1} still imply
    \[
        \htop\bigl(\Phi,G_{\mathcal K}(\Phi)\bigr)
        \geq
        \liminf_{i\to\infty} h_{\mu_i}(\phi_1).
    \]

    Indeed, one only needs to remove the initial orbit segment used to enter the given open set and the auxiliary reference blocks associated with the points $x_i$. The resulting tree construction consists solely of the  \(\Gamma_{t_i}\)-blocks. The estimates used in the proof of the entropy lower bound remain valid, with fewer error terms.
\end{remark}

\section{Proofs of Theorems \ref{thm:A} and \ref{thm:B} and applications}\label{sec:6}
In this section, we apply the abstract lower-bound criterion established in Theorem~\ref{thm:5.1} to geometric Lorenz attractors and singular hyperbolic attractors. The main task is to verify, in the corresponding generic settings, the hypotheses required there: horseshoe approximation, entropy-expansiveness, and the existence of non-decreasing mixing horseshoes. Combining this lower bound with the upper bound proved in Section~\ref{sec:3}, we obtain Theorems~\ref{thm:A} and~\ref{thm:B}. We then use these entropy formulae to recover the previous multifractal results on irregular sets and level sets.

\begin{lemma}\label{lem:lorenz_residual}
    For every $r \geq 2$, there exists a residual subset $\mathcal{R}^r \subset \cX^r\left(M^3\right)$ such that, if $\Lambda$ is a geometric Lorenz attractor of $X \in \mathcal{R}^r$, then the following statements hold.
    \begin{enumerate}
        \item \cite{shi2023multifractal}*{Corollary 4.14(1)} $\left(\Phi,\Lambda\right)$ satisfies the horseshoe approximation property.
        \item \cite{pacifico2025equilibrium}*{Theorem 2.17}\cite{liao2013entropy}*{Proposition 2.4} $\left(\Phi,\Lambda\right)$ is entropy-expansive, and hence satisfies the uniform separation property.
        \item \cite{shi2023multifractal}*{Proposition 2.6} $\Lambda$ is a homoclinic class. Moreover, any two periodic orbits contained in $\Lambda$ are homoclinically related.
        \item $\Lambda$ has two periodic points with incommensurate periods.
    \end{enumerate}
\end{lemma}

\begin{lemma}\label{lem:sectional_residual}
    There exists a residual subset $\mathcal{R} \subset \cX^1\left(M\right)$ such that, if $\Lambda$ is a non-trivial singular hyperbolic attractor of $X \in \mathcal{R}$, then the following statements hold.
    \begin{enumerate}
        \item \cite{shi2023multifractal}*{Corollary 4.14(2)} $\left(\Phi,\Lambda\right)$ satisfies the horseshoe approximation property.
        \item \cite{pacifico2025equilibrium}*{Theorem 2.17}\cite{liao2013entropy}*{Proposition 2.4} $\left(\Phi,\Lambda\right)$ is entropy-expansive, and hence satisfies the uniform separation property.
        \item \cite{crovisier2021robust}*{Theorem B} $\Lambda$ is a homoclinic class, any two periodic orbits contained in $\Lambda$ are hyperbolic and homoclinically related.
        \item $\Lambda$ has two periodic points with incommensurate periods.
    \end{enumerate}
\end{lemma}
\begin{remark}
Item (2) of each of Lemmas~\ref{lem:lorenz_residual} and~\ref{lem:sectional_residual} holds for every singular hyperbolic set with hyperbolic singularities, without any residual assumption.
\end{remark}

We only prove Lemma \ref{lem:sectional_residual}(4). The proof of Lemma \ref{lem:lorenz_residual}(4) is similar.

\begin{proof}[Proof of Lemma \ref{lem:sectional_residual}(4).]
 Choose a countable basis $\{U_n\}_{n\ge 1}$ of $M$, and let $\{O_n\}_{n\ge 1}$ be an enumeration of all finite unions of elements of this basis.

For $i,j,p,q\in\mathbb N^+$, let $\mathcal G_{i,j,p,q}$ be the set of vector fields $X\in\mathfrak X^1(M)$ such that $X$ has two distinct hyperbolic periodic orbits $\gamma_i,\gamma_j$ satisfying
$$
\gamma_i\subset O_i,\quad
\gamma_j\subset O_j,\quad
\text{$\gamma_i$ and $\gamma_j$ are homoclinically related},\quad
pT(\gamma_i)\neq qT(\gamma_j).
$$
The set $\mathcal G_{i,j,p,q}$ is open because
\begin{itemize}
    \item hyperbolic periodic orbits admit continuations under $C^1$ perturbation.
    \item the transverse homoclinic intersections persist under $C^1$ perturbation. 
    \item periods depend continuously on the vector field. Therefore the strict inequality $pT(\gamma_i(Y))\neq qT(\gamma_j(Y))$ is an open property.
\end{itemize}
Set
\begin{align*}
    \mathcal N_{i,j,p,q}
&=\operatorname{Int}\bigl(\mathfrak X^1(M)\setminus \mathcal G_{i,j,p,q}\bigr)  \\
&=\{ X| \text{$X$ has a $C^1$-neighborhood $\mathcal{U}$, such that $\mathcal{U} \cap \mathcal{G}_{i,j,p,q}=\emptyset$} \},
\end{align*}
where $\operatorname{Int}(\cdot)$ denotes the interior with
respect to the $C^1$ topology on $\mathfrak{X}^1(M)$. Then $\mathcal G_{i,j,p,q}\cup \mathcal N_{i,j,p,q}$ is open and dense in $\mathfrak X^1(M)$. Hence
$$
\mathcal R_4
=
\bigcap_{i,j,p,q\in\mathbb N^+}
\bigl(\mathcal G_{i,j,p,q}\cup \mathcal N_{i,j,p,q}\bigr)
$$
is residual.

Let $\mathcal R_{123}$ be a residual subset of $\mathfrak X^1(M)$ on which the first three assertions of Lemma~\ref{lem:sectional_residual} hold. Now set
$\mathcal R=\mathcal R_{123}\cap\mathcal R_4.$
Then $\mathcal R$ is residual. We claim that $\mathcal R$ is the desired residual subset.

\medskip
\noindent\textbf{Claim.}
    For every $X\in\mathcal R$ and every non-trivial singular hyperbolic attractor $\Lambda$ of $X$, any two distinct periodic orbits contained in $\Lambda$ have incommensurate periods.

\begin{proof}[Proof of the claim]
Let $X\in\mathcal R$, and let $\Lambda$ be a non-trivial singular hyperbolic attractor of $X$. By \cite{crovisier2021robust}*{Theorem B}, periodic points are dense in $\Lambda$. Since $\Lambda$ is non-trivial, let $\gamma_i,\gamma_j\subset\Lambda$ be any two distinct periodic orbits. Since $X\in\mathcal R_{123}$, the third assertion holds for $X$. Hence $\gamma_{i}$ and $\gamma_{j}$ are hyperbolic and homoclinically related.

Suppose, by contradiction, that their periods are commensurate. Then there exist $p,q\in\mathbb N^+$ such that $pT(\gamma_{i})=qT(\gamma_{j}).$
Since $\gamma_{i}$ and $\gamma_{j}$ are hyperbolic, choose small isolating neighborhoods $V_{\gamma_{i}},V_{\gamma_{j}}$ of them. Taking finite unions of basis elements, choose $O_{\gamma_{i}},O_{\gamma_{j}}$ such that
$$
\gamma_{i} \subset O_{\gamma_{i}}\subset V_{\gamma_{i}},\qquad
\gamma_{j}\subset O_{\gamma_{j}}\subset V_{\gamma_{j}}.
$$
Thus the only periodic orbit of $X$ contained in $O_{\gamma_{i}}$ is $\gamma_{i}$, and the only periodic orbit of $X$ contained in $O_{\gamma_{j}}$ is $\gamma_{j}$. Since $pT(\gamma_{i})=qT(\gamma_{j})$, we have
$X\notin\mathcal G_{i,j,p,q}$. Since $X\in\mathcal R_4$, it follows that
$X\in\mathcal N_{i,j,p,q}$.
Therefore there exists a $C^1$-neighborhood $\mathcal U$ of $X$ such that
$$
\mathcal U\cap\mathcal G_{i,j,p,q}=\varnothing.
$$

For any $\varepsilon>0$, by a bump function argument, we can choose a positive smooth function $\rho_{\varepsilon}$, such that $\rho_{\epsilon}=1+\varepsilon$ on $\gamma_{i}$ and $\rho_{\epsilon}=1$ on $\gamma_{j}$, and the $C^r$-distance between $\rho_{\varepsilon}$ and $1$ goes to $0$ as $\varepsilon \rightarrow 0$. Put
$$
Y_{\varepsilon}=\rho_{\varepsilon} X.
$$
We have $Y_{\varepsilon}\in\mathcal U$ when $\varepsilon$ is small enough, and $Y_{\varepsilon} \overset{C^r}{\rightarrow} X$ as $\varepsilon \rightarrow 0$. Since this is a positive time change, $X$ and $Y_{\varepsilon}$ have the same orbit sets. In particular, no new periodic orbit is created in the isolating neighborhoods $V_{\gamma_{i}}$ and $V_{\gamma_{j}}$. Moreover, the orbit sets of $\gamma_i$ and $\gamma_j$, their hyperbolicity, and their homoclinic relation are preserved. The only change is that the period of $\gamma_i$ is modified, while the period of $\gamma_j$ is unchanged:
$$
T_{Y_{\varepsilon}}(\gamma_{i})=\frac{T_X(\gamma_{i})}{1+\varepsilon},
\qquad
T_{Y_{\varepsilon}}(\gamma_{j})=T_X(\gamma_{j}).
$$
Choosing $\varepsilon\neq0$ small enough so that $pT_{Y_{\varepsilon}}(\gamma_{i})\neq qT_{Y_{\varepsilon}}(\gamma_{j})$, 
we obtain
$Y_{\varepsilon}\in\mathcal G_{i,j,p,q}\cap\mathcal U$, 
which is a contradiction.

Thus $T(\gamma_{i})/T(\gamma_{j})\notin\mathbb Q$. This proves the claim. 
\end{proof}

The claim immediately implies that every non-trivial singular hyperbolic attractor of every $X\in\mathcal R$ contains two periodic orbits with incommensurate periods. This completes the proof. 
\end{proof}

\begin{remark}
    The claim gives a slightly stronger conclusion than Lemma \ref{lem:sectional_residual}(4): for $X\in\mathcal R$, the periodic orbits contained in any non-trivial singular hyperbolic attractor of $X$ are pairwise incommensurate.
\end{remark} 
The homoclinic relatedness and incommensurate-period property established above yield the non-decreasing sequence of mixing horseshoes required below. Indeed, Lemma~\ref{lem:larger_horseshoe} allows us to merge homoclinically related horseshoes, while the presence of periodic orbits with incommensurate periods guarantees that each enlarged horseshoe is topologically mixing.

\subsection{Proofs of Theorem \ref{thm:A} and Theorem \ref{thm:B}}
\begin{proof}[Proofs of Theorem \ref{thm:A} and Theorem \ref{thm:B}]
    We only prove Theorem \ref{thm:A}; the proof of Theorem \ref{thm:B} is similar.
    By Theorem \ref{thm:upper_htop}, one has
    \[\htop\left(\Phi, G_{\mathcal{K}}(\Phi) \right) \leq \inf\left\{h_\mu\left(\phi_1\right) :\mu \in \mathcal{K}\right\}.\]

    Let $U \subseteq M$ be a nonempty open set with $U \cap \Lambda \ne \emptyset$. It suffices to show that
    \begin{equation}\label{eq:saturated_lower_bound}
        \htop\left(\Phi, G_{\mathcal{K}}(\Phi) \cap \Trans\left(\Phi, \Lambda\right) \cap U \right) \geq \inf\left\{h_\mu\left(\phi_1\right) :\mu \in \mathcal{K}\right\}.
    \end{equation}    
    Let $H =\inf \left\{h_\mu\left(\phi_1\right) : \mu \in \mathcal{K}\right\}$. Assume $H > 0$. Similar to the argument in \cite{pfister2007topological}*{Page 944}, there exists a sequence of measures $\left\{\mu_i \in \M(\Phi, \Lambda):i\in\Nb^+\right\}$ in $\mathcal{K}$ such that
    \begin{enumerate}
        \item $\overline{\left\{\mu_i\right\}_{i=j}^\infty} = \mathcal{K}$,
        \item $D\left(\mu_i, \mu_{i+1}\right) \to 0$ as $i \to \infty$,
        \item $\inf\left\{h_{\mu_i}\left(\phi_1\right):i \in \Nb^+\right\} = H$.
    \end{enumerate}
    For any $0 < \eta < H/5$, since $\Lambda$ has the horseshoe approximation property, there exist a sequence of horseshoes $\left\{\Lambda_i:i\in \Nb^+\right\}$ and a sequence of measures $\left\{\nu_i \in \M_{\mathrm{erg}}\left(\Phi, \Lambda_{i}\right) : i \in \Nb^+\right\}$ such that
    \begin{enumerate}
        \item $\bigcap_{j\in\Nb^+} \overline{\left\{\nu_i\right\}_{i=j}^\infty} = \mathcal{K}$,
        \item $D\left(\nu_i, \nu_{i+1}\right) \to 0$ as $i \to \infty$,
        \item $\inf\left\{h_{\nu_i}\left(\phi_1\right):i \in \Nb^+\right\} > H - \eta$.
    \end{enumerate}

    By \cite{arroyo2007dynamical}*{Theorem A}, the set of periodic points is dense in $\Lambda$. We choose a countable dense family of periodic orbits
    $ \{ p_{n} \}_{n \geq 1} \subset \Lambda$.
    By Lemma \ref{lem:lorenz_residual}, each pair of periodic orbits is homoclinically related and there exist two periodic points whose periods are incommensurate. Then, by applying Lemma \ref{lem:larger_horseshoe} inductively, we may enlarge $\left\{\Lambda_i:i\in \Nb^+\right\}$ to a non-decreasing sequence of horseshoes $\left\{\widetilde{\Lambda}_i:i\in \Nb^+\right\}$, such that each $\tilde{\Lambda_n}$ contains the previously chosen measure-support horseshoes, periodic points $p_{1}$ to $p_{n}$, and two periodic orbits with incommensurate periods. Moreover, this implies
    $$ \overline{\bigcup_{n=1}^{\infty}\widetilde{\Lambda}_{n}}=\Lambda.   $$

    By Theorem \ref{incommensurates_implies_mixing} and Theorem \ref{mixing_implies_specification}, each $\left(\widetilde{\Lambda}_i,\Phi\right)$ is topologically mixing and hence satisfies the specification property for any $i \in \Nb^+$. Then (\ref{eq:saturated_lower_bound}) follows from Theorem \ref{thm:5.1} by the arbitrariness of $\eta$.

    When $H=0$, the same approximation and construction, without the quantitative entropy restrictions, give $G_{\mathcal{K}}(\Phi) \cap \Trans(\Phi,\Lambda)\cap U \neq \emptyset$, as explained at the end of the proof of Theorem~\ref{thm:5.1}. The entropy equality then follows from Theorem~\ref{thm:upper_htop}.
    
\end{proof}

\subsection{A relative saturated entropy formula}

The proof above uses the horseshoe approximation property to approximate the prescribed invariant measures by ergodic measures supported on a non-decreasing family of mixing horseshoes. The same idea admits a more general formulation. Namely, suppose that a non-decreasing family of compact invariant subsystems has the specification property. We first show that every invariant measure giving full mass to their union can be approximated, both in the weak* topology and in entropy, by ergodic measures supported on one of these subsystems. This approximation result, combined with Theorem~\ref{thm:5.1}, yields a relative saturated entropy
formula.
\begin{lemma}\label{lem:relative_entropy_density}
    Let $\Phi=(\phi_t)_{t\in\Rb}$ be a continuous flow on a compact metric
    space $\Lambda$, and assume that
    \[
        h_{\mathrm{top}}(\phi_1,\Lambda)<+\infty.
    \]
    Let $\{\Lambda_n\}_{n\in\Nb^+}$ be a non-decreasing sequence of compact
    $\Phi$-invariant subsets of $\Lambda$ such that
    $(\Lambda_n,\Phi)$ satisfies the specification property for every $n\in\Nb^+$. For any $\mu \in \mathcal{M}(\Phi,\Lambda)$ satisfying $\mu(\cup_{n \geq 1}\Lambda_{n})=1$ and every
    $\ve>0$, there exist $n\in\Nb^+$ and
    $\nu\in\Merg(\Phi,\Lambda_n)$ such that
    \[
        D(\mu,\nu)<\ve
        \qquad\text{and}\qquad
        h_\nu(\phi_1)>h_\mu(\phi_1)-\ve.
    \]
\end{lemma}

\begin{proof}
     Fix $\ve>0$. Take $a_n:=\mu(\Lambda_n)$. Since $\Lambda_n$ is non-decreasing and
    $\mu(\cup_{n \geq 1}\Lambda_{n})=1$,
    we have $\lim_{n \to \infty}a_n=1$. We first consider the case \(a_n<1\) for every \(n\ge1\).

    For every $n \in \Nb^+$ with $a_n>0$, define
    \[
        \mu_n:=\frac{1}{a_n}\mu|_{\Lambda_n}.
    \]
    Since $\Lambda_n$ is $\Phi$-invariant, $\mu_n\in\M(\Phi,\Lambda_n)$.
     Define similarly
    \[
        \widehat\mu_n
        :=
        \frac{1}{1-a_n}\mu|_{\Lambda\setminus\Lambda_n}.
    \]
    Then
    \[
        \mu=a_n\mu_n+(1-a_n)\widehat\mu_n.
    \]
    Since $\diam\Lambda\le1$, the Wasserstein metric satisfies
    $D(\alpha,\beta)\le1$ for all $\alpha,\beta\in\M(\Lambda)$. Hence
    \[
        D(\mu,\mu_n)
        \le
        (1-a_n)D(\widehat\mu_n,\mu_n)
        \le 1-a_n,
    \]
    and therefore $\lim_{n \to \infty}\mu_n=\mu$.

    By the affinity of measure-theoretical entropy,
    \[
        h_\mu(\phi_{1})
        =
        a_nh_{\mu_n}(\phi_{1})
        +(1-a_n)h_{\widehat\mu_n}(\phi_{1}).
    \]
    Since $h_{\widehat\mu_n}(\phi_{1})\le h_{\mathrm{top}}(\phi_{1},\Lambda)$, we obtain
    \[
        h_{\mu_n}(\phi_{1})
        \ge
        \frac{h_\mu(\phi_{1})-(1-a_n)h_{\mathrm{top}}(\phi_{1},\Lambda)}{a_n}.
    \]
    Letting $n \rightarrow \infty$, we have
    
    $$\liminf_{n\to\infty}h_{\mu_n}(\phi_{1})\ge h_\mu(\phi_{1}).$$
     Thus, we can choose $n\in\Nb^+$ sufficiently large such that
    \begin{equation}\label{eq:relative_choice_n}
        D(\mu,\mu_n)<\frac{\ve}{3}
        \qquad\text{and}\qquad
        h_{\mu_n}(\phi_{1})>h_\mu(\phi_{1})-\frac{\ve}{3}.
    \end{equation}

    We claim that $(\Lambda_n,\phi_{1})$ satisfies the specification property.
    Indeed, fix $\rho>0$, and let $T_\rho>0$ be given by the specification
    property of $(\Lambda_n,\Phi)$. Put
    \[
        N_\rho:=\lceil T_\rho\rceil.
    \]
    Given finitely many $\phi_{1}$-orbit segments $(x_1,\ell_1),\ldots,(x_k,\ell_k)$, define
    \[
        s_1:=0,
        \qquad
        s_{i+1}:=s_i+\ell_i+N_\rho
        \quad (1\le i<k).
    \]
    Applying the specification property of $\Phi$ to the corresponding
    flow-orbit segments, we obtain $y\in\Lambda_n$ such that
    \[
        d\bigl(\phi_{1}^{s_i+j}(y),\phi_{1}^j(x_i)\bigr)<\rho
    \]
    for every $1\le i\le k$ and every $0\le j<\ell_i$.
    Hence $(\Lambda_n,\phi_{1})$ has the specification property.

    Consequently, $(\Lambda_n,\phi_{1})$ has the approximate product property.
    By the entropy-density theorem of Pfister--Sullivan
    \cite{pfister2005large}*{Theorem 2.1}, there exists $\widetilde\nu\in\Merg(\phi_{1},\Lambda_n)$
    such that
    \begin{equation}\label{eq:relative_entropy_dense}
        D(\widetilde\nu,\mu_n)<r
        \qquad\text{and}\qquad
        h_{\widetilde\nu}(\phi_1)>h_{\mu_n}(\phi_1)-\frac{\ve}{3},
    \end{equation}
    where $r>0$ will be chosen below.

    Since $\tau$ is continuous and $\tau\mu_n=\mu_n$, there exists $r>0$
    such that
    \begin{equation}\label{eq:relative_tau_continuity}
         D(\alpha,\mu_n)<r
        \quad\Longrightarrow\quad
        D(\tau\alpha,\mu_n)<\frac{\ve}{3}.
    \end{equation}
    Let $\nu:=\tau\widetilde\nu$.
    Applying Lemma \ref{lem:tau} to the compact invariant space
    $\Lambda_n$, we have
    \[
        \nu\in\Merg(\Phi,\Lambda_n)
        \qquad\text{and}\qquad
        h_\nu(\phi_1)=h_{\widetilde\nu}(\phi_1).
    \]
    By \eqref{eq:relative_choice_n}, \eqref{eq:relative_entropy_dense}, and
    \eqref{eq:relative_tau_continuity},
    \[
        D(\mu,\nu)
        \le
        D(\mu,\mu_n)+D(\mu_n,\tau\widetilde\nu)
        <
        \frac{2\ve}{3}
        <
        \ve.
    \]
    Moreover,
    \[
        h_\nu(\phi_1)
        =
        h_{\widetilde\nu}(\phi_1)
        >
        h_{\mu_n}(\phi_1)-\frac{\ve}{3}
        >
        h_\mu(\phi_1)-\frac{2\ve}{3}
        >
        h_\mu(\phi_1)-\ve.
    \]
    This completes the proof. If \(a_n=1\) for some \(n\), the conclusion similarly follows from the entropy-density theorem of Pfister–Sullivan applied to \(\phi_1|_{\Lambda_n}\), followed by the averaging map \(\tau\).
\end{proof}

\begin{corollary}\label{cor:relative_saturated}
    Let $(\Lambda,\Phi)$ be topologically transitive, asymptotically
    entropy-expansive, and have finite topological entropy. Suppose that
    $\{\Lambda_n\}_{n\geq 1}$ is a non-decreasing sequence of compact
    $\Phi$-invariant subsets of $\Lambda$ such that
    \[
        \overline{\bigcup_{n\geq 1}\Lambda_n}=\Lambda
    \]
    and $(\Lambda_n,\Phi)$ satisfies the specification property for every
    $n\geq 1$.

    Then, for every nonempty compact connected set
    \[
        \mathcal{K}\subseteq \{ \mu \in \mathcal{M}(\Phi,\Lambda): \mu \left(\bigcup_{n \geq 1} \Lambda_{n} \right)=1 \}
    \]
    and every nonempty relatively open set $U\subseteq\Lambda$, one has
    \[
        \htop\left(
        \Phi,
        G_{\mathcal{K}}(\Phi)\cap \Trans(\Phi,\Lambda)\cap U
        \right)
        =
        \inf\{h_\mu(\phi_1):\mu\in \mathcal{K}\}.
    \]

    In particular, suppose that $(\Lambda,\Phi)$ is asymptotically entropy-expansive,
    has finite topological entropy, and satisfies the specification property.
    Then, for every nonempty compact connected set
    $\mathcal{K}\subseteq\mathcal{M}(\Phi,\Lambda)$ and every nonempty relatively open set $U\subseteq\Lambda$,
    \[
        \htop\left(
        \Phi,
        G_{\mathcal{K}}(\Phi)\cap\Trans(\Phi,\Lambda)\cap U
        \right)
        =
         \inf\{h_\mu(\phi_1):\mu\in \mathcal{K}\}.
    \]
\end{corollary}
\begin{proof}
    Let $H=\inf\{h_\mu(\phi_1):\mu\in \mathcal{K}\}$. Choose a sequence
    $\{\mu_i\}_{i\geq 1}\subseteq \mathcal{K}$ such that
    \[
        \bigcap_{j\geq 1}\overline{\{\mu_i:i\geq j\}}=\mathcal{K}
        \qquad\text{and}\qquad
        D(\mu_i,\mu_{i+1})\to 0.
    \]
    By Lemma \ref{lem:relative_entropy_density}, for every $i\geq 1$ there
    exist $n_i\geq i$ and
    $\nu_i\in\mathcal{M}_{\rm erg}(\Phi,\Lambda_{n_i})$ such that
    \[
        D(\mu_i,\nu_i)<\frac1i
        \qquad\text{and}\qquad
        h_{\nu_i}(\phi_1)>h_{\mu_i}(\phi_1)-\frac1i.
    \]
    Define recursively
$m_1:=n_1$, $m_{i+1}:=\max\{n_{i+1},m_i+1\}$ and set $\widetilde{\Lambda}_i:=\Lambda_{m_i}$. We have
\[
\nu_i\in\mathcal{M}_{\mathrm{erg}}(\Phi,\Lambda_{n_i})
\subseteq
\mathcal{M}_{\mathrm{erg}}(\Phi,\widetilde{\Lambda}_i).
\]
Moreover, $\{\widetilde{\Lambda}_i\}_{i\geq1}$ is non-decreasing and
\[
\overline{\bigcup_{i\geq1}\widetilde{\Lambda}_i}=\Lambda.
\]
   The sequences
    $\{\nu_i\}_{i\geq 1}$ and $\{\widetilde{\Lambda}_i\}_{i\geq1}$ satisfy the assumptions of Theorem
    \ref{thm:5.1}. Hence
    \[
        \htop\left(
        \Phi,
        G_{\mathcal{K}}(\Phi)\cap\Trans(\Phi,\Lambda)\cap U
        \right)\geq H.
    \]
    The reverse inequality follows from Theorem \ref{thm:upper_htop}.
\end{proof}

\subsection{Multifractal application} \label{subsec:5.2}
As an application of Theorems \ref{thm:A} and \ref{thm:B}, we provide an alternative proof of the multifractal results established in \cite{shi2023multifractal}*{Theorem A, Theorem B}.

Throughout this subsection, $\Phi=(\phi_t)_{t\in\mathbb R}$ denotes the flow generated by the vector field $X$. We use the notation $I_g$ and $R_g(a)$ introduced in the Introduction for the irregular set and the level set associated with
$g\in C(\Lambda,\mathbb{R})$, respectively.

\begin{theorem}[\cite{shi2023multifractal}*{Theorem A, Theorem B}]\label{thm:2.1}
    There exists a Baire residual subset $\mathcal{R}^r \subset \mathcal{X}^r(M^3)$,  $(r\in\mathbb{N}_{\geqslant 2})$, and a Baire residual subset $\mathcal{R} \subset \mathcal{X}^1(M)$ so that, if $\Lambda$ is a geometric Lorenz attractor of $X\in\mathcal{R}^r$ or a singular hyperbolic attractor of $X\in\mathcal{R}$, and $g\in C(\Lambda,\mathbb{R})$, then either:
    \begin{enumerate}
        \item $I_g$ is empty and $\int g \, d\mu = \int g \, d\nu$ for all $\mu,\nu\in\M(\Phi,\Lambda)$, or
        \item $I_g$ is a residual subset of $\Lambda$ and $h_{\mathrm{top}}(\Phi,I_g) = h_{\mathrm{top}}(\Phi,\Lambda)$.
    \end{enumerate}
    Moreover, if $I_g$ is non-empty then for any $a\in\mathbb{R}$ satisfying
    \[
        \inf_{\mu\in\M(\Phi,\Lambda)} \int g \, d\mu < a < \sup_{\mu\in\M(\Phi,\Lambda)} \int g \, d\mu,
    \]
    the level set $R_g(a)$ is dense in $\Lambda$ and
    \[
        h_{\mathrm{top}}(\Phi,R_g(a)) = \sup \left\{h_\mu(\phi_{1}) : \int g \, d\mu = a, \mu\in\M(\Phi,\Lambda)\right\}.
    \]
\end{theorem}

\begin{proof}
    We only prove the case of geometric Lorenz attractors; the case of singular hyperbolic attractors is similar. We take $\mathcal{R}^{r}$ as in Theorem \ref{thm:A} and assume $\Lambda$ is a  geometric Lorenz attractor of $X \in \mathcal{R}^{r}$.

    Suppose first that $I_g$ is empty. Suppose by contradiction, that there exist $\mu,\nu \in \M(\Phi,\Lambda)$ such that
    $$ \int g \, d\mu > \int g \, d\nu.  $$
    We construct $\mathcal{K}=[\mu,\nu]=\{  \lambda\mu+(1-\lambda)\nu:\lambda\in [0,1] \}$, which is a compact connected subset. By Theorem \ref{thm:A}, $G_{\mathcal{K}}(\Phi)$ is  nonempty. Take $x \in G_{\mathcal{K}}(\Phi)$. since
    $$  \frac{1}{t}\int_{0}^{t}g(\phi_{s}(x))ds=\int g d(\frac{1}{t}\int_{0}^{t}\delta_{\phi_{s}(x)}ds),  $$
    we have
    $$  \liminf_{t \rightarrow \infty}\frac{1}{t}\int_{0}^{t}g(\phi_{s}(x))ds \leq \int g \, d\nu<\int g \, d\mu \leq \limsup_{t \rightarrow \infty}\frac{1}{t}\int_{0}^{t}g(\phi_{s}(x))ds.$$
    This implies that $\lim\limits_{T\rightarrow \infty}\frac{1}{T}\int_{0}^{T}g(\phi_{t}(x)) dt$ does not exist. Hence, $x \in I_{g}$, which contradicts the assumption that $I_{g}$ is empty.

    Suppose now that $I_g$ is nonempty. We first prove that $h_{\mathrm{top}}(\Phi,I_g) = h_{\mathrm{top}}(\Phi,\Lambda)$. For $x \in I_{g}$, there exist $\mu_{0},\mu_{1}\in V_{\Phi}(x)$ such that
    $$ \int g \, d \mu_{0} \neq \int g \, d\mu_{1} .$$
    Fix $\varepsilon>0$. By the variational principle, there is $\eta \in \mathcal{M}(\Phi,\Lambda)$, such that
    $$ h_{\eta}(\phi_{1})>h_{\mathrm{top}}(\Phi,\Lambda)-\varepsilon.  $$
    Then there is $\mu_{0}$ or $\mu_{1}$, denoted by $\mu_{0}$ without loss of generality, such that
    $$\int g \, d \eta \neq \int g \, d \mu_{0}  $$
    Let $\nu=(1-a)\eta+a \mu_{0}$, where $a$ is small enough such that
    $$ h_{\nu}(\phi_{1})>h_{\mathrm{top}}(\Phi,\Lambda)-2\varepsilon.  $$
    Let
    $ \mathcal{K}=[\eta,\nu]=\{ (1-\lambda)\eta+\lambda\nu :0 \leq \lambda \leq 1 \} $, which is a compact connected subset. Since $\mu \mapsto \int g d\mu$ is not constant on $\mathcal{K}$, by an argument similar to the previous one, we can obtain  $G_{\mathcal{K}}(\Phi) \subset I_{g}$. By Theorem \ref{thm:A} and the affine property of entropy map,
    $$  h_{\mathrm{top}}(\Phi,I_{g}) \geq h_{\mathrm{top}}(\Phi,G_{\mathcal{K}}(\Phi)\cap \mathrm{Trans}(\Phi,\Lambda))=\inf_{\mu \in \mathcal{K}} h_{\mu}(\phi_{1}) \geq h_{\mathrm{top}}(\Phi,\Lambda)-2\varepsilon. $$
    Letting $\varepsilon \rightarrow 0$, we get $h_{\mathrm{top}}(\Phi,I_{g}) \geq h_{\mathrm{top}}(\Phi,\Lambda)$. On the other hand, $I_{g} \subset \Lambda$ implies $h_{\mathrm{top}}(\Phi,I_{g}) \leq h_{\mathrm{top}}(\Phi,\Lambda)$, which in turn implies
    $$   h_{\mathrm{top}}(\Phi,I_{g}) = h_{\mathrm{top}}(\Phi,\Lambda)$$

   Now we prove that $I_{g}$ is a residual set of $\Lambda$. Take $\{B(\mu_{0},\frac{1}{n}),n \in \mathbb{N}^{+}\}$ and $\{B(\mu_{1},\frac{1}{n}),n \in \mathbb{N}^{+}\}$ as two sequences of weak star neighborhoods that shrink to $\mu_{0},\mu_{1}$, respectively. When $n$ is large enough, it satisfies
    $$  \int g \, d \eta_{0} \neq \int g \, d \eta_{1}, \ for \ \eta_{0} \in B(\mu_{0},\frac{1}{n}),\  \eta_{1}\in B(\mu_{1},\frac{1}{n}). $$
    We define
    \begin{align*}
        \mathcal{R}_{n} 
        & =\{ x \in \Lambda:\mathcal{E}_{T}(x,\Phi) \in B(\mu_{0},\frac{1}{n}) \ \text{for arbitrarily large} \ T, \ \mathcal{E}_{T}(x,\Phi) \in B(\mu_{1},\frac{1}{n}) \ \text{for arbitrarily large} \ T \} \\
        & =\bigcap_{m=1}^{\infty}(\{ x:\exists T \geq m, \mathcal{E}_{T}(x,\Phi) \in B(\mu_{0},\frac{1}{n}) \} \cap \{ x:\exists T \geq m, \mathcal{E}_{T}(x,\Phi) \in B(\mu_{1},\frac{1}{n})   \}).
    \end{align*}
    It is clear that $\mathcal{R}_{n}$ is a $G_{\delta}$ set. Consider the compact connected set $\mathcal{K}=[\mu_{0},\mu_{1}] \subset \M(\Phi,\Lambda)$. By Theorem \ref{thm:A}, for any non-empty open set $U \subset \Lambda$,
    $$ G_{\mathcal{K}}(\Phi) \cap \mathrm{Trans}(\Phi,\Lambda) \cap U \neq \emptyset.  $$
    Since $G_{\mathcal{K}}(\Phi) \subset \mathcal{R}_{n}$,  $\mathcal{R}_{n}$ is dense in $\Lambda$, and consequently $\mathcal{R}_{n}$ is residual. We have
    $$  \mathcal{R}_{\mathrm{irr}}=\bigcap_{n=1}^{\infty} \mathcal{R}_{n}  $$
    is residual. For any $x \in \mathcal{R}_{\mathrm{irr}}$, its empirical measures along different subsequences approach measures with different g-integrals. Thus $\mathcal{R}_{\mathrm{irr}} \subset I_{g}$. Therefore, $I_{g}$ is residual.

    Finally, we consider the entropy of level sets.  Define
    \[  \mathcal{M}_{a}=\{  \mu \in \mathcal{M}(\Phi,\Lambda):\int g \, d\mu=a \}.  \]
    When $\inf_{\mu \in \M(\Phi,\Lambda)} \int g \, d\mu < a < \sup_{\mu\in \M(\Phi,\Lambda)} \int g \, d\mu$, $\mathcal{M}_{a}$ is non-empty.
    Take any $\mu \in \mathcal{M}_{a}$. Then $G_{\mu}(\Phi) \subset R_{g}(a)$.
    Applying Theorem \ref{thm:A} gives
    $$ h_{\mathrm{top}}(\Phi,R_{g}(a)) \geq  h_{\mathrm{top}}(\Phi,G_{\mu}(\Phi) \cap \mathrm{Trans}(\Phi,\Lambda))=h_{\mu}(\phi_{1}).  $$
    Taking the supremum over $\mu \in \mathcal{M}_{a}$ yields
    $$  h_{\mathrm{top}}(\Phi,R_{g}(a)) \geq \sup_{\mu \in \mathcal{M}_{a}}h_{\mu}(\phi_{1}).$$
    
    On the other hand, if $x \in R_{g}(a)$, we have $\int g \, d\nu=a$ for every $\nu \in V_{\Phi}(x)$ and hence $V_{\Phi}(x) \subset \mathcal{M}_{a}$. For every $\nu \in V_{\phi_{1}}(x)$, Lemma~\ref{lem:tau} gives $\tau \nu \in \mathcal{M}_{a}$ and 
    $$ h_{\nu}(\phi_{1})=h_{\tau \nu}(\phi_{1}) \leq \sup_{\mu \in \mathcal{M}_{a}} h_{\mu}(\phi_{1}).  $$
    Hence
    $$ R_{g}(a) \subseteq QR\left(\sup_{\mu \in \mathcal{M}_{a}}h_{\mu}(\phi_{1})\right).  $$
    The upper bound 
    $$  h_{\mathrm{top}}(\Phi,R_{g}(a)) \leq \sup_{\mu \in \mathcal{M}_{a}}h_{\mu}(\phi_{1})  $$
    follows from Lemmas~\ref{lem:Bowen} and \ref{lem:htop}.
    Therefore, $  h_{\mathrm{top}}(\Phi,R_{g}(a)) = \sup_{\mu \in \mathcal{M}_{a}}h_{\mu}(\phi_{1}).  $
    The density of the level sets follows directly from the local version of Theorem \ref{thm:A}.
\end{proof}

\section{Proofs of Theorem \ref{thm:C} and Theorem \ref{thm:D}}
We now prove Theorems~\ref{thm:C} and~\ref{thm:D}.
The common mechanism is a local mass inequality near a singularity.

\begin{lemma}\label{lem:star}
    Let $X\in\cX^1(M)$, let $\Phi=(\phi_t)_{t\in\Rb}$ be the flow generated by $X$, and let $\Lambda\subset M$ be a compact $\Phi$-invariant set. Assume that $\sigma\in\Sing(X)$ is hyperbolic with $\dim E^u(\sigma)=1$, and that $p\in\Lambda$ is a hyperbolic periodic point of $X$ such that  $W^u(\sigma) \backslash \left\{\sigma\right\} \subset W^s\left(\Orb(p)\right) $. Then there exists
    \begin{itemize}
        \item an open neighborhood $U^*$ of $\sigma$ and a closed tubular neighborhood $V^*$ of $\Orb(p)$ with $\overline{U^*}\cap V^* =\emptyset$;
        \item a constant $L=L(X,\sigma,p)>0$,
    \end{itemize}
    such that for any $x \in M$, either $V_{\Phi}(x) = \left\{\delta_\sigma\right\}$ or
    \begin{equation}\label{eq:star}
        \mu(V^*) \geq L \mu(U^*)
    \end{equation}
    for some $\mu \in V_\Phi(x)$.
\end{lemma}

We follow the construction of \cite{shi2025connecting}*{Section~4.3} and merely indicate how the inequality \eqref{eq:star} arises before invoking any limit measure.

\begin{proof}
    By the Hartman--Grobman/Shilnikov adapted-coordinate construction
    \cite{shilnikov1998methods}*{Theorems~2.4 and~2.5}, choose a $C^1$ chart on a neighborhood
    \[U(\sigma):=\left\{(x,y) = \left(x_{1},\dots,x_{d-1},y\right) \in \Rb^{d-1}\times \Rb:\left\lVert x\right\rVert \leq 1, \left\lVert y\right\rVert \leq 1 \right\}\]
    of $\sigma$ such that $\sigma=(0,0)\in\Rb^{d-1}\times\Rb$, $\Wsloc(\sigma)=\{y=0\}$ and $\Wuloc(\sigma)=\{x=0\}$ as \cite{shi2025connecting}*{Claim 4.7}. Write $u_1<0$ for the stable contraction exponent in the $x$-direction in these adapted coordinates.
    For any $0 < \delta < 1$, denote
    \[\Sigma^{\pm} = \left\{(x,\pm 1) \in U(\sigma):\left\lVert x\right\rVert \leq 1 \right\},\]
    \[\Sigma^{\pm}(\delta) = \left\{(x,\pm 1) \in \Sigma^{\pm}:\left\lVert x\right\rVert \leq \delta \right\},\]
    and $U(\sigma,\delta) = \left\{(x,y) \in U(\sigma): \left\lVert x\right\rVert \leq \delta, \left\lvert y\right\rvert \leq \delta \right\}$. For any $z = (x_z,y_z) \in U(\sigma,\delta)$ with $x_z \ne 0$, $y_z > 0$, there exist $T_1 < 0 < T_2$ such that $\phi_{T_1}(z) \in \partial U(\sigma) \setminus \Sigma^{\pm}$, $\phi_{T_2}(z) \in \Sigma^{+}$, and $\phi_t(z) \in \Int U(\sigma)$ for any $t \in (T_1,T_2)$ (see \cite{shi2025connecting}*{Claim 4.8}). Choose a $C^1$ cross-section
    \[\Sigma_p :=\left\{(r,s) \in \Rb^k \times \Rb^{d-1-k}: \left\lVert r\right\rVert \leq 1, \left\lVert s\right\rVert \leq 1 \right\} \]
    at $p$ and two constants $\kappa > 1$, $T_0 > 0$ such that the following statements hold (see \cite{shi2025connecting}*{Claim 4.9}).
    \begin{enumerate}
        \item The Poincar\'e return map
              \[R:\Sigma_p' := \left\{(r,s) \in \Sigma_p:\left\lVert s\right\rVert \leq \kappa^{-1}\right\} \to \Sigma_p \]
              is well-defined.
        \item For any $z \in \Sigma_p'$, let $T_z > 0$ be the first return time of $z$ to $\Sigma_p$. Then $T_z \geq T_0$ and
              \[\left(\bigcup_{z \in \Sigma_p'} \phi_{\left[0,T_z\right]}(z)\right) \cap U(\sigma) = \emptyset.\]
    \end{enumerate}
    Let $V^* = \bigcup_{z \in \Sigma_p'} \phi_{\left[0,T_z\right]}(z)$ which is a closed neighborhood of $\Orb(p)$. There exists $0 < \delta_0 < 1$ such that the Poincar\'e maps $P^\pm:\Sigma^\pm(\delta_0) \to \Sigma_p$ are well-defined. For any $z = (x_z,y_z) \in U(\sigma,\delta_0)$ with $x_z \ne 0, y_z > 0$, let $T_1 < 0 < T_2$ be defined as above and $w = (x_w, 1) = \phi_{T_2}(z) \in \Sigma^+(\delta_0)$. Then for the point $P^+(w)=(r_w,s_w) \in \Sigma_p'$, by the estimates in \cite{shi2025connecting}*{Claim 4.10}, there exists a constant $K>0$, independent of $w$, such that
    \[
        \|s_w\|\leq K\|x_w\|.
    \]
    After decreasing $\delta_0$ if necessary, assume that $K\|x_w\|<1$. The orbit of $P^+(w)$ remains inside $V^*$ for at least
    \[
        T_p\geq\Bigl(\tfrac{-\log(K\|x_w\|)}{\log\kappa}-2\Bigr)T_0
    \]
    units of time (see \cite{shi2025connecting}*{Claim 4.10}), while it spent at most
    \[
        T_2-T_1\leq\frac{\log\|x_w\|}{u_1}=\frac{|\log\|x_w\||}{|u_1|}
    \]
    units of time inside $U(\sigma,\delta_0)$. The same conclusion holds for $z = (x_z,y_z) \in U(\sigma,\delta_0)$ with $x_z \ne 0$ and $y_z < 0$. Choose $\delta_1\in(0,\delta_0]$ small enough such that
    \begin{equation}\label{eq:tratio}
        T_p\geq L\cdot(T_2-T_1),\qquad L:=\frac{-u_1 T_0}{2\log\kappa}>0.
    \end{equation}
    Let 
    $$U^* = \operatorname{Int}U(\sigma,\delta_1)=\{ (x,y) \in U(\sigma):\Vert x \Vert < \delta_{1},\Vert y \Vert <\delta_1  \}.  $$ 
    Then for any $z = (x_z,y_z) \in U^* \setminus \left\{\Wsloc(\sigma) \cup \Wuloc(\sigma)\right\}$, its orbit spends at most $T_2-T_1$ units of time inside $U^*$, wanders outside $U^*$ and enters $V^*$, and spends at least $T_p$ units of time inside $V^*$ until exiting from $V^*$ with $T_p\geq L\cdot(T_2-T_1)$.

    If $x \in \Wsloc(\sigma)$, then $x$ converges to $\sigma$ and hence $V_\Phi(x) = \left\{\delta_\sigma\right\}$. 
    If $x \notin \Wsloc(\sigma)$ and there exists $T>0$ such that
    $\phi_t(x) \notin U^*$ for all $t\geq T$, then, for any $\mu\in V_\Phi(x)$, choose $t_n\to+\infty$ such that $\mathcal E_{t_n}(x)\to\mu$. Since
    \[
    \mathcal E_{t_n}(x)(U^*)\leq \frac{T}{t_n}\longrightarrow 0,
    \]
    and $U^*$ is open, the Portmanteau theorem gives
    \[
    \mu(U^*)\leq\liminf_{n\to\infty}
    \mathcal E_{t_n}(x)(U^*)=0.
    \]
    Consequently, $\mu(U^*)=0$ and hence \eqref{eq:star} holds.  If $x \notin \Wsloc(\sigma)$ and $\phi_t(x) \in U^*$ for any $t \geq T$ for some $T$ large enough, then $\phi_T(x) \in W^s_{\mathrm{loc}}(\sigma)$ by \cite{shi2025connecting}*{Claim 4.7(2)}. Hence $x$ belongs to the global stable set of $\sigma$, and consequently $V_\Phi(x)=\{\delta_\sigma\}$. If $x \notin \Wsloc(\sigma)$ and $x$ enters $U^*$ and leaves from $U^*$ for infinitely many times, then decompose the forward orbit of $x$ into
    disjoint time intervals
    \[
        0\leq T_1^1<T_2^1<T_3^1<T_4^1<T_1^2<T_2^2<\cdots
    \]
    where $[T_1^m,T_2^m]$ is the $m$-th maximal sojourn of $\phi_t(x)$ in $U^*$ and
    $[T_3^m,T_4^m]$ is the first maximal sojourn in $V^*$ following it. By
    \eqref{eq:tratio},
    \[
        T_4^m-T_3^m\geq L\cdot(T_2^m-T_1^m)
    \]
    for any $m \in \Nb^+$. Let $\mu \in V_\Phi(x)$ be an accumulation point of $\left\{\cE_{T^m_4}(x):m \in \Nb^+\right\}$ and assume $\mu = \lim_{l \to \infty} \cE_{T^{m_l}_4}(x)$. Then
    \[
        \mu(U^*)\leq\liminf_{l\to\infty}\cE_{T_4^{m_l}}(x)(U^*)\leq\liminf_{l\to\infty}\frac{1}{T^{m_l}_4}\sum_{i=1}^{m_l}(T_2^i-T_1^i),
    \]
    \[
        \mu(V^*)\geq\limsup_{l\to\infty}\cE_{T_4^{m_l}}(x)(V^*)\geq\limsup_{l\to\infty}\frac{1}{T^{m_l}_4}\sum_{i=1}^{m_l}(T_4^i-T_3^i) \geq L \mu(U^*).
    \]
\end{proof}

\begin{proof}[Proof of Theorem \ref{thm:C} and Theorem \ref{thm:D}]
    By \cite{shi2025connecting}*{Proof of Theorem B}, there exists a dense subset $\mathcal{D}^r\subset\cX^r(M^3)$ such that, if $\Lambda$ is a non-trivial geometric Lorenz attractor of $X \in \mathcal{D}^r$, then the hypotheses of Lemma \ref{lem:star} hold.
    
    Fix such an $X$ and $\Lambda$. Since a geometric Lorenz attractor has positive topological entropy, the classical variational principle yields a measure $\mu \in \M(\Phi,\Lambda)$ with $h_{\mu}(\phi_{1})>0$. In particular, $\mu \ne \delta_\sigma$. 
    
    Let $U^*$,$V^*$, and $L$ be given by Lemma~\ref{lem:star}. Choose $\theta \in (0,1)$, such that, for 
    $$\mu_\theta:= \theta   \delta_\sigma + (1-\theta)\mu \neq \delta_{\sigma},$$
    one has
    $$\mu_\theta\left(U^*\right) > \frac{1}{1+L} \quad \text{and} \quad h_{\mu_\theta}(\phi_{1}) = (1-\theta)h_\mu(\phi_{1}) > 0 .$$
    
    We claim that $G_{\mu_\theta}(\Phi)=\emptyset$. Otherwise, take
    $x\in G_{\mu_\theta}(\Phi)$. Since
    $V_\Phi(x)=\{\mu_\theta\}\neq\{\delta_\sigma\}$, Lemma~\ref{lem:star}
    gives
    \[
        \mu_\theta(V^*)\geq L\mu_\theta(U^*)
        >\frac{L}{1+L}.
    \]
    As $U^*\cap V^*=\emptyset$, this implies
    \[
        1=\mu_\theta(M)
        \geq \mu_\theta(U^*)+\mu_\theta(V^*)
        >1,
    \]
    a contradiction. This proves Theorem~\ref{thm:C}.

   An analogous conclusion for singular hyperbolic attractors follows from the dense-part proof of \cite{shi2025connecting}*{Theorem B.1}(see also \cite{shi2025connecting}*{Remark 4.6}): there is a dense subset $\mathcal{D}\subset\cX^1(M)$ such that, if $\Lambda$ is a non-trivial singular hyperbolic attractor of $X \in \mathcal{D}$ satisfying $\Sing(\Lambda)\neq\emptyset$ and $\dim(E^{cu})=2$, then the hypothesis of Lemma \ref{lem:star} holds. Using an argument similar to the previous one, we can prove Theorem \ref{thm:D}.
\end{proof}

\appendix
\section{A suspension-flow example}
\label{app:suspension-example}
The purpose of this appendix is to show that the inclusion
\[
    \phi_{[0,1]}G_{\mathcal K}(\phi_1)
    \subseteq G_{\tau\mathcal K}(\Phi)
\]
in Lemma~\ref{lem:tau}(\ref{item:4.1.8}) can be strict. We demonstrate this using a constant-roof suspension flow. 

Let $(\Sigma,\sigma)$ be the two-sided full shift on two symbols,
and let $\Phi=(\varphi_t)_{t\in\mathbb{R}}$ be the suspension flow
with constant roof function $\rho\equiv1$. The suspension space is
\[
\Sigma^\rho=\Sigma\times[0,1]/\sim,
\]
where $\sim$ is the equivalence relation generated by
$(y,1)\sim(\sigma(y),0)$ for every $y\in\Sigma$.
The time-one map acts by
\[
\varphi_1(y,s)=(\sigma(y),s).
\]
We will exhibit a nonempty compact connected set
$\mathcal{K}\subseteq\mathcal{M}(\varphi_1,\Sigma^\rho)$
for which
\[
\varphi_{[0,1]}G_{\mathcal{K}}(\varphi_1)
\subsetneq G_{\tau\mathcal{K}}(\Phi).
\]
Although Lemma~\ref{lem:tau}(\ref{item:4.1.7}) gives the exact relation
\[
\tau\bigl(V_{\varphi_1}(x)\bigr)=V_\Phi(x)
\]
between the accumulation sets of discrete-time and continuous-time
empirical measures, it does not yield a corresponding equality
between their saturated sets. The example relies on the following
distinction: the time-one map preserves the fibre coordinate,
whereas the flow moves continuously through this coordinate.

\begin{example}[Strictness in Lemma \ref{lem:tau}(\ref{item:4.1.8})]
    Take a $\sigma$-ergodic measure $\mu \in\mathcal{M}_{\mathrm{erg}}(\sigma,\Sigma)$
    and set
    \[
        \mathcal{K} = \left\{\mu \otimes\mathrm{Leb}\right\},
    \]
    where $\mathrm{Leb}$ denotes the normalized Lebesgue measure on $[0,1)$. Then $\mathcal{K}$ is a nonempty compact connected subset of both $\mathcal{M}(\phi_{1},\Sigma_\rho)$ and $\mathcal{M}(\Phi,\Sigma_\rho)$.

    \medskip
    \noindent\textbf{Claim.}
    $G_{\tau \mathcal{K}}(\Phi)\neq\emptyset$, yet $\phi_{[0,1]}G_{\mathcal{K}}(\phi_1)=\emptyset$.
    
    \begin{proof}
        Since $\rho\equiv 1$, it is easy to check that
        \[
            V_{\phi_1}((y,s))
            = \left\{\nu\otimes\delta_s : \nu\in V_\sigma(y)\right\}
            = V_\sigma(y)\times\{\delta_s\}.
        \]

        The element in $\mathcal{K}$ has marginal $\mathrm{Leb}$ on the fibre coordinate, while $V_{\phi_1}((y,s))$ has marginal $\delta_s\neq\mathrm{Leb}$ for any $s\in[0,1)$.
        It follows that $V_{\phi_1}((y,s))\neq\mathcal K$ for every
        $(y,s)\in\Sigma_{\rho}$, and consequently
        \[
        G_{\mathcal K}(\phi_1)=\emptyset
        \quad\text{and}\quad
        \phi_{[0,1]}G_{\mathcal K}(\phi_1)=\emptyset .
        \]

        For the flow, one computes
        \[
            \cE_t((y,s),\Phi)
            = \frac{1}{t}\int_0^t \delta_{\phi_u((y,s))}\,du.
        \]
        As $t\to\infty$, the fibre marginal converges to $\mathrm{Leb}$, and
        the base marginal converges along the same subsequences as $\cE_{\lfloor t\rfloor}(y,\sigma)$. More precisely,
        \[
            V_\Phi(y,s)
            = \bigl\{\nu\otimes\mathrm{Leb} : \nu\in V_\sigma(y)\bigr\}
            = \tau\bigl(V_{\phi_1}(y,s)\bigr).
        \]
        Then
        \[
            G_{\mathcal{K}}(\Phi)
            = \{(y,s)\in\Sigma_\rho : V_\sigma(y)=\left\{\mu\right\} \}
            = G_{\mu}(\sigma)\times[0,1).
        \]
        By Birkhoff Ergodic Theorem, we have $G_{\mu}(\sigma)\neq\emptyset$. Therefore, $G_{\mathcal{K}}(\Phi)\neq\emptyset$. Note that $\mu \otimes\mathrm{Leb}$ is $\Phi$-invariant and satisfies $\tau(\mu \otimes\mathrm{Leb})=\mu \otimes\mathrm{Leb}$, so $\tau \mathcal{K}=\mathcal{K}$.
        We have shown
        \[
            \phi_{[0,1]}G_{\mathcal{K}}(\phi_1)
            = \emptyset \subsetneq G_{\mathcal{K}}(\Phi)
            = G_{\tau \mathcal{K}}(\Phi),
        \]
        which demonstrates that the inclusion in Lemma \ref{lem:tau}(\ref{item:4.1.8}) can be strict.
    \end{proof}
\end{example}

\subsection*{Acknowledgements}
We would like to thank Y. Shi for useful suggestions
and comments.

\bibliography{ref}

@book {shilnikov1998methods,
    AUTHOR = {Shilnikov, Leonid P. and Shilnikov, Andrey L. and Turaev,
              Dmitry V. and Chua, Leon O.},
     TITLE = {Methods of qualitative theory in nonlinear dynamics. {P}art
              {I}},
    SERIES = {World Scientific Series on Nonlinear Science. Series A:
              Monographs and Treatises},
    VOLUME = {4},
      NOTE = {With the collaboration of Sergey Gonchenko (Sections 3.7 and
              3.8), Oleg Sten\cprime kin (Section 3.9 and Appendix A) and
              Mikhail Shashkov (Sections 6.1 and 6.2)},
 PUBLISHER = {World Scientific Publishing Co., Inc., River Edge, NJ},
      YEAR = {1998},
     PAGES = {xxiv+392},
      ISBN = {981-02-3382-5},
   MRCLASS = {37-02 (34C23 34C28 34D45 37D45 37Gxx)},
  MRNUMBER = {1691840},
MRREVIEWER = {Stanis\l aw\ S\c edziwy},
       DOI = {10.1142/9789812798596},
       URL = {https://doi.org/10.1142/9789812798596},
}

@article {arroyo2007dynamical,
    AUTHOR = {Arroyo, Aubin and Pujals, Enrique R.},
     TITLE = {Dynamical properties of singular-hyperbolic attractors},
   JOURNAL = {Discrete Contin. Dyn. Syst.},
  FJOURNAL = {Discrete and Continuous Dynamical Systems},
    VOLUME = {19},
      YEAR = {2007},
    NUMBER = {1},
     PAGES = {67--87},
      ISSN = {1078-0947,1553-5231},
   MRCLASS = {37D30 (37C10 37C70 37D10)},
  MRNUMBER = {2318274},
MRREVIEWER = {Joan\ Carles\ Tatjer},
       DOI = {10.3934/dcds.2007.19.67},
       URL = {https://doi.org/10.3934/dcds.2007.19.67},
}

@book{hirsch1977invariant,
    AUTHOR = {Hirsch, M. W. and Pugh, C. C. and Shub, M.},
     TITLE = {Invariant manifolds},
    SERIES = {Lecture Notes in Mathematics},
    VOLUME = {583},
 PUBLISHER = {Springer-Verlag, Berlin-New York},
      YEAR = {1977},
     PAGES = {ii+149},
   MRCLASS = {58F15 (58F10)},
  MRNUMBER = {501173},
MRREVIEWER = {M.\ C.\ Irwin},
}

@article {crovisier2021robust,
    AUTHOR = {Crovisier, Sylvain and Yang, Dawei},
     TITLE = {Robust transitivity of singular hyperbolic attractors},
   JOURNAL = {Math. Z.},
  FJOURNAL = {Mathematische Zeitschrift},
    VOLUME = {298},
      YEAR = {2021},
    NUMBER = {1-2},
     PAGES = {469--488},
      ISSN = {0025-5874,1432-1823},
   MRCLASS = {37C70 (37C20 37D30)},
  MRNUMBER = {4257095},
MRREVIEWER = {C.\ M.\ Carballo},
       DOI = {10.1007/s00209-020-02618-1},
       URL = {https://doi.org/10.1007/s00209-020-02618-1},
}

@article{bowen1972entropy,
  author     = {Bowen, Rufus},
  title      = {Entropy-expansive maps},
  journal    = {Trans. Amer. Math. Soc.},
  fjournal   = {Transactions of the American Mathematical Society},
  volume     = {164},
  year       = {1972},
  pages      = {323--331},
  issn       = {0002-9947,1088-6850},
  mrclass    = {28.70 (54.00)},
  mrnumber   = {285689},
  mrreviewer = {Karl\ Sigmund},
  doi        = {10.2307/1995978},
  url        = {https://doi.org/10.2307/1995978}
}

@article{bowen1972expansive,
  author     = {Bowen, Rufus and Walters, Peter},
  title      = {Expansive one-parameter flows},
  journal    = {J. Differential Equations},
  fjournal   = {Journal of Differential Equations},
  volume     = {12},
  year       = {1972},
  pages      = {180--193},
  issn       = {0022-0396,1090-2732},
  mrclass    = {54H20},
  mrnumber   = {341451},
  mrreviewer = {L.\ King},
  doi        = {10.1016/0022-0396(72)90013-7},
  url        = {https://doi.org/10.1016/0022-0396(72)90013-7}
}

@article{bowen1973topological,
  author     = {Bowen, Rufus},
  title      = {Topological entropy for noncompact sets},
  journal    = {Trans. Amer. Math. Soc.},
  fjournal   = {Transactions of the American Mathematical Society},
  volume     = {184},
  year       = {1973},
  pages      = {125--136},
  issn       = {0002-9947,1088-6850},
  mrclass    = {28A65 (54H20)},
  mrnumber   = {338317},
  mrreviewer = {Karl\ Sigmund},
  doi        = {10.2307/1996403},
  url        = {https://doi.org/10.2307/1996403}
}

@incollection{brin1983local,
  author     = {Brin, M. and Katok, A.},
  title      = {On local entropy},
  booktitle  = {Geometric dynamics ({R}io de {J}aneiro, 1981)},
  series     = {Lecture Notes in Math.},
  volume     = {1007},
  pages      = {30--38},
  publisher  = {Springer, Berlin},
  year       = {1983},
  isbn       = {3-540-12336-9},
  mrclass    = {58F11 (28D20)},
  mrnumber   = {730261},
  mrreviewer = {R.\ Dil\~ao},
  doi        = {10.1007/BFb0061408},
  url        = {https://doi.org/10.1007/BFb0061408}
}

@article{feng2012variational,
  author     = {Feng, De-Jun and Huang, Wen},
  title      = {Variational principles for topological entropies of subsets},
  journal    = {J. Funct. Anal.},
  fjournal   = {Journal of Functional Analysis},
  volume     = {263},
  year       = {2012},
  number     = {8},
  pages      = {2228--2254},
  issn       = {0022-1236,1096-0783},
  mrclass    = {37B40 (28D20 54H20)},
  mrnumber   = {2964682},
  mrreviewer = {Vaughn\ Climenhaga},
  doi        = {10.1016/j.jfa.2012.07.010},
  url        = {https://doi.org/10.1016/j.jfa.2012.07.010}
}

@book{fisher2019hyperbolic,
  author     = {Fisher, Todd and Hasselblatt, Boris},
  title      = {Hyperbolic flows},
  series     = {Zurich Lectures in Advanced Mathematics},
  publisher  = {EMS Publishing House, Berlin},
  year       = {2019},
  pages      = {xiv+723},
  isbn       = {978-3-03719-200-9},
  mrclass    = {37-02 (37A30 37A35 37D20 37D40)},
  mrnumber   = {3972204},
  mrreviewer = {Miguel\ Paternain},
  doi        = {10.4171/200},
  url        = {https://doi.org/10.4171/200}
}

@article{guckenheimer1979structural,
  author     = {Guckenheimer, John and Williams, R. F.},
  title      = {Structural stability of {L}orenz attractors},
  journal    = {Inst. Hautes \'Etudes Sci. Publ. Math.},
  fjournal   = {Institut des Hautes \'Etudes Scientifiques. Publications
                Math\'ematiques},
  number     = {50},
  year       = {1979},
  pages      = {59--72},
  issn       = {0073-8301,1618-1913},
  mrclass    = {58F13 (58F25)},
  mrnumber   = {556582},
  mrreviewer = {Zbigniew\ Nitecki},
  url        = {http://www.numdam.org/item?id=PMIHES_1979__50__59_0}
}

@article{hou2021entropy,
  title         = {Entropy of irregular points that are not uniformly hyperbolic},
  author        = {Xiaobo Hou and Xueting Tian},
  year          = {2021},
  eprint        = {2111.08325},
  archiveprefix = {arXiv},
  primaryclass  = {math.DS},
  url           = {https://arxiv.org/abs/2111.08325}
}

@article{li2005robustly,
  author     = {Li, Ming and Gan, Shaobo and Wen, Lan},
  title      = {Robustly transitive singular sets via approach of an extended
                linear {P}oincar\'e{} flow},
  journal    = {Discrete Contin. Dyn. Syst.},
  fjournal   = {Discrete and Continuous Dynamical Systems. Series A},
  volume     = {13},
  year       = {2005},
  number     = {2},
  pages      = {239--269},
  issn       = {1078-0947,1553-5231},
  mrclass    = {37D30 (37C10)},
  mrnumber   = {2152388},
  mrreviewer = {Carlos\ A.\ Morales},
  doi        = {10.3934/dcds.2005.13.239},
  url        = {https://doi.org/10.3934/dcds.2005.13.239}
}

@article{liang2016variational,
  author     = {Liang, Chao and Liao, Gang and Sun, Wenxiang and Tian,
                Xueting},
  title      = {Variational equalities of entropy in nonuniformly hyperbolic
                systems},
  journal    = {Trans. Amer. Math. Soc.},
  fjournal   = {Transactions of the American Mathematical Society},
  volume     = {369},
  year       = {2017},
  number     = {5},
  pages      = {3127--3156},
  issn       = {0002-9947,1088-6850},
  mrclass    = {37B40 (37C40 37D25)},
  mrnumber   = {3605966},
  mrreviewer = {Serge\ E.\ Troubetzkoy},
  doi        = {10.1090/tran/6780},
  url        = {https://doi.org/10.1090/tran/6780}
}

@article{liao2013entropy,
  author     = {Liao, Gang and Viana, Marcelo and Yang, Jiagang},
  title      = {The entropy conjecture for diffeomorphisms away from
                tangencies},
  journal    = {J. Eur. Math. Soc. (JEMS)},
  fjournal   = {Journal of the European Mathematical Society (JEMS)},
  volume     = {15},
  year       = {2013},
  number     = {6},
  pages      = {2043--2060},
  issn       = {1435-9855,1435-9863},
  mrclass    = {37G25 (37B10 37B40 37C20)},
  mrnumber   = {3120734},
  mrreviewer = {Robert\ Sta\'nczy},
  doi        = {10.4171/JEMS/413},
  url        = {https://doi.org/10.4171/JEMS/413}
}

@article{metzger2008sectional,
  author   = {Metzger, R. and Morales, C.},
  title    = {Sectional-hyperbolic systems},
  journal  = {Ergodic Theory Dynam. Systems},
  fjournal = {Ergodic Theory and Dynamical Systems},
  volume   = {28},
  year     = {2008},
  number   = {5},
  pages    = {1587--1597},
  issn     = {0143-3857,1469-4417},
  mrclass  = {37D30 (37C10 37C27)},
  mrnumber = {2449545},
  doi      = {10.1017/S0143385707000995},
  url      = {https://doi.org/10.1017/S0143385707000995}
}

@article{morales1999singular,
  author     = {Morales, C. A. and Pacifico, M. J. and Pujals, E. R.},
  title      = {Singular hyperbolic systems},
  journal    = {Proc. Amer. Math. Soc.},
  fjournal   = {Proceedings of the American Mathematical Society},
  volume     = {127},
  year       = {1999},
  number     = {11},
  pages      = {3393--3401},
  issn       = {0002-9939,1088-6826},
  mrclass    = {37D05 (37C70)},
  mrnumber   = {1610761},
  mrreviewer = {Rodrigo\ Bamon},
  doi        = {10.1090/S0002-9939-99-04936-9},
  url        = {https://doi.org/10.1090/S0002-9939-99-04936-9}
}

@article{morales2004robust,
  author     = {Morales, C. A. and Pacifico, M. J. and Pujals, E. R.},
  title      = {Robust transitive singular sets for 3-flows are partially
                hyperbolic attractors or repellers},
  journal    = {Ann. of Math. (2)},
  fjournal   = {Annals of Mathematics. Second Series},
  volume     = {160},
  year       = {2004},
  number     = {2},
  pages      = {375--432},
  issn       = {0003-486X,1939-8980},
  mrclass    = {37D30 (37C10 37D45 37E99)},
  mrnumber   = {2123928},
  mrreviewer = {Ant\'onio\ Pumari\~no},
  doi        = {10.4007/annals.2004.160.375},
  url        = {https://doi.org/10.4007/annals.2004.160.375}
}

@article{pacifico2022bowen,
  author     = {Pacifico, Maria Jos\'{e} and Sanhueza, Diego},
  title      = {On {B}owen's entropy inequality and almost specification for
                flows},
  journal    = {J. Differential Equations},
  fjournal   = {Journal of Differential Equations},
  volume     = {335},
  year       = {2022},
  pages      = {43--68},
  issn       = {0022-0396,1090-2732},
  mrclass    = {37B40},
  mrnumber   = {4454276},
  mrreviewer = {Zeya\ Mi},
  doi        = {10.1016/j.jde.2022.06.028},
  url        = {https://doi.org/10.1016/j.jde.2022.06.028}
}

@article{pacifico2025equilibrium,
  author   = {Pacifico, Maria Jose and Yang, Fan and Yang, Jiagang},
  title    = {Equilibrium states for the classical {L}orenz attractor and
              sectional-hyperbolic attractors in higher dimensions},
  journal  = {Duke Math. J.},
  fjournal = {Duke Mathematical Journal},
  volume   = {174},
  year     = {2025},
  number   = {10},
  pages    = {1901--2010},
  issn     = {0012-7094,1547-7398},
  mrclass  = {37D35 (37C40)},
  mrnumber = {4941878},
  doi      = {10.1215/00127094-2024-0067},
  url      = {https://doi.org/10.1215/00127094-2024-0067}
}

@book{palis1982geometric,
  author     = {Palis, Jr., Jacob and de Melo, Welington},
  title      = {Geometric theory of dynamical systems},
  note       = {An introduction,
                Translated from the Portuguese by A. K. Manning},
  publisher  = {Springer-Verlag, New York-Berlin},
  year       = {1982},
  pages      = {xii+198},
  isbn       = {0-387-90668-1},
  mrclass    = {58-01 (34C35 58F09 58Fxx)},
  mrnumber   = {669541},
  mrreviewer = {Russell\ B.\ Walker}
}

@article{pfister2005large,
  author     = {Pfister, C.-E. and Sullivan, W. G.},
  title      = {Large deviations estimates for dynamical systems without the
                specification property. {A}pplications to the
                {$\beta$}-shifts},
  journal    = {Nonlinearity},
  fjournal   = {Nonlinearity},
  volume     = {18},
  year       = {2005},
  number     = {1},
  pages      = {237--261},
  issn       = {0951-7715,1361-6544},
  mrclass    = {37A45 (11K55 37A25 37A50 37C35 37D35 60F10)},
  mrnumber   = {2109476},
  mrreviewer = {Yves\ Lacroix},
  doi        = {10.1088/0951-7715/18/1/013},
  url        = {https://doi.org/10.1088/0951-7715/18/1/013}
}

@article{pfister2007topological,
  author     = {Pfister, C.-E. and Sullivan, W. G.},
  title      = {On the topological entropy of saturated sets},
  journal    = {Ergodic Theory Dynam. Systems},
  fjournal   = {Ergodic Theory and Dynamical Systems},
  volume     = {27},
  year       = {2007},
  number     = {3},
  pages      = {929--956},
  issn       = {0143-3857,1469-4417},
  mrclass    = {37B40},
  mrnumber   = {2322186},
  mrreviewer = {Godofredo\ Iommi},
  doi        = {10.1017/S0143385706000824},
  url        = {https://doi.org/10.1017/S0143385706000824}
}

@article{shi2023multifractal,
  author     = {Shi, Yi and Tian, Xueting and Varandas, Paulo and Wang,
                Xiaodong},
  title      = {On multifractal analysis and large deviations of singular
                hyperbolic attractors},
  journal    = {Nonlinearity},
  fjournal   = {Nonlinearity},
  volume     = {36},
  year       = {2023},
  number     = {10},
  pages      = {5216--5251},
  issn       = {0951-7715,1361-6544},
  mrclass    = {37D25 (37A50 37D30 37D35)},
  mrnumber   = {4646022},
  mrreviewer = {Yuan-Ling\ Ye},
  doi        = {10.1088/1361-6544/ace491},
  url        = {https://doi.org/10.1088/1361-6544/ace491}
}

@article{shi2025connecting,
  title   = {A {$C^r$}-CONNECTING LEMMA FOR {L}ORENZ ATTRACTORS AND ITS APPLICATION ON THE SPACE OF ERGODIC MEASURES},
  doi     = {10.1017/S147474802510114X},
  journal = {Journal of the Institute of Mathematics of Jussieu},
  author  = {Shi, Yi and Tian, Xueting and Wang, Xiaodong},
  year    = {2025},
  pages   = {1–44}
}

@book{villani2009optimal,
  author     = {Villani, C\'edric},
  title      = {Optimal transport},
  series     = {Grundlehren der mathematischen Wissenschaften [Fundamental
                Principles of Mathematical Sciences]},
  volume     = {338},
  note       = {Old and new},
  publisher  = {Springer-Verlag, Berlin},
  year       = {2009},
  pages      = {xxii+973},
  isbn       = {978-3-540-71049-3},
  mrclass    = {49-02 (28A75 37J50 49Q20 53C23 58E30)},
  mrnumber   = {2459454},
  mrreviewer = {Dario\ Cordero-Erausquin},
  doi        = {10.1007/978-3-540-71050-9},
  url        = {https://doi.org/10.1007/978-3-540-71050-9}
}

@article{williams1979structure,
  author     = {Williams, R. F.},
  title      = {The structure of {L}orenz attractors},
  journal    = {Inst. Hautes \'Etudes Sci. Publ. Math.},
  fjournal   = {Institut des Hautes \'Etudes Scientifiques. Publications
                Math\'ematiques},
  number     = {50},
  year       = {1979},
  pages      = {73--99},
  issn       = {0073-8301,1618-1913},
  mrclass    = {58F13 (58F25)},
  mrnumber   = {556583},
  mrreviewer = {Zbigniew\ Nitecki},
  url        = {http://www.numdam.org/item?id=PMIHES_1979__50__73_0}
}

@article {afraimovich1977origin,
    AUTHOR = {Afraimovich, V. S. and Bykov, V. V. and Shil'nikov, L. P.},
     TITLE = {The origin and structure of the {Lorenz} attractor},
   JOURNAL = {Dokl. Akad. Nauk SSSR},
  FJOURNAL = {Doklady Akademii Nauk SSSR},
    VOLUME = {234},
      YEAR = {1977},
    NUMBER = {2},
     PAGES = {336--339},
      NOTE = {(Russian)},
}

@incollection {guckenheimer1976strange,
    AUTHOR = {Guckenheimer, John},
     TITLE = {A strange, strange attractor},
 BOOKTITLE = {The {Hopf} Bifurcation and Its Applications},
    SERIES = {Applied Mathematical Sciences},
    VOLUME = {19},
 PUBLISHER = {Springer-Verlag},
   ADDRESS = {New York},
      YEAR = {1976},
     PAGES = {368--381},
       DOI = {10.1007/978-1-4612-6374-6_25},
       URL = {https://doi.org/10.1007/978-1-4612-6374-6_25},
}

@article {lorenz1963deterministic,
    AUTHOR = {Lorenz, Edward N.},
     TITLE = {Deterministic nonperiodic flow},
   JOURNAL = {J. Atmospheric Sci.},
  FJOURNAL = {Journal of the Atmospheric Sciences},
    VOLUME = {20},
      YEAR = {1963},
    NUMBER = {2},
     PAGES = {130--141},
      ISSN = {0022-4928,1520-0469},
       DOI = {10.1175/1520-0469(1963)020<0130:DNF>2.0.CO;2},
       URL = {https://doi.org/10.1175/1520-0469(1963)020<0130:DNF>2.0.CO;2},
}

@article {zhu2008indices,
    AUTHOR = {Zhu, Shengzhi and Gan, Shaobo and Wen, Lan},
     TITLE = {Indices of singularities of robustly transitive sets},
   JOURNAL = {Discrete Contin. Dyn. Syst.},
  FJOURNAL = {Discrete and Continuous Dynamical Systems},
    VOLUME = {21},
      YEAR = {2008},
    NUMBER = {3},
     PAGES = {945--957},
      ISSN = {1078-0947,1553-5231},
   MRCLASS = {37D30},
       DOI = {10.3934/dcds.2008.21.945},
       URL = {https://doi.org/10.3934/dcds.2008.21.945},
}

@article{huang2019transitively,
  author  = {Huang, Yu and Tian, Xueting and Wang, Xiaoyi},
  title   = {Transitively-saturated property, {Banach} recurrence and
             {Lyapunov} regularity},
  journal = {Nonlinearity},
  volume  = {32},
  year    = {2019},
  number  = {7},
  pages   = {2721--2757},
  doi     = {10.1088/1361-6544/ab090c}
}

@article{WenWen2020,
  author  = {Wen, Xiao and Wen, Lan},
  title   = {No-shadowing for singular hyperbolic sets with a singularity},
  journal = {Discrete and Continuous Dynamical Systems},
  volume  = {40},
  number  = {10},
  pages   = {6043--6059},
  year    = {2020},
  doi     = {10.3934/dcds.2020258},
  url     = {https://doi.org/10.3934/dcds.2020258}
}

@article{sumi2015specification,
  author  = {Sumi, Naoya and Varandas, Paulo and Yamamoto, Kenichiro},
  title   = {Specification and partial hyperbolicity for flows},
  journal = {Dynamical Systems},
  volume  = {30},
  number  = {4},
  pages   = {501--524},
  year    = {2015},
  doi     = {10.1080/14689367.2015.1081380},
  url     = {https://doi.org/10.1080/14689367.2015.1081380}
}
\end{document}